\documentclass{amsart}
\usepackage{amssymb}
\usepackage{amscd}
\usepackage{hyperref}
\usepackage{pdfpages}

\makeatletter
\newcommand*\rel@kern[1]{\kern#1\dimexpr\macc@kerna}
\newcommand*\widebar[1]{%
  \begingroup
  \def\mathaccent##1##2{%
    \rel@kern{0.8}%
    \overline{\rel@kern{-0.8}\macc@nucleus\rel@kern{0.2}}%
    \rel@kern{-0.2}%
  }%
  \macc@depth\@ne
  \let\math@bgroup\@empty \let\math@egroup\macc@set@skewchar
  \mathsurround\z@ \frozen@everymath{\mathgroup\macc@group\relax}%
  \macc@set@skewchar\relax
  \let\mathaccentV\macc@nested@a
  \macc@nested@a\relax111{#1}%
  \endgroup
}
\makeatother

\newcommand{\isomto}{\overset{\thicksim}{\longrightarrow}}

\newcommand{\longto}{\longrightarrow}

\def\bQ{{\mathbf{Q}}}  
  \def\bF{{\mathbf{F}}}

\def\cO{{\mathcal{O}}}

 \def\cU{{\mathcal{U}}} 
  \def\cL{{\mathcal{L}}}

 \def\fm{{\mathfrak{m}}} 
 \def\fM{{\mathfrak{M}}} \def\fL{{\mathfrak{L}}}

\DeclareMathOperator\Gal{{Gal}}

\DeclareMathOperator\length{{length}}
\DeclareMathOperator\BC{{BC}}

\renewcommand{\det}{\mathrm{det}}

\usepackage[OT2,T1]{fontenc}
\DeclareSymbolFont{cyrletters}{OT2}{wncyr}{m}{n}
\DeclareMathSymbol{\Sha}{\mathalpha}{cyrletters}{"58}

\DeclareMathOperator\id{{id}}

\DeclareMathOperator\tr{{tr}}

\DeclareMathOperator\Fitt{{Fitt}}

\def\sep{{\rm sep}}

\def\H{{\rm H}}
\def\U{{\rm U}}
\def\C{{\rm C}}

\def\Frob{{\rm Frob}}

\def\tors{{\rm tors}}
\def\phi{\varphi}

\makeatletter
\def\swappedhead#1#2#3{%
  \thmnumber{\@upn{\the\thm@headfont#2\@ifnotempty{#1}{.~}}}%
  \thmname{#1}
  \thmnote{ {\the\thm@notefont(#3)}}}
\makeatother

\newcounter{mainthmcounter}

\theoremstyle{plain}
\newtheorem{theorem}[subsection]{Theorem}
\newtheorem{corollary}[subsection]{Corollary}
\newtheorem{lemma}[subsection]{Lemma}
\newtheorem{proposition}[subsection]{Proposition}

\theoremstyle{definition}

\newtheoremstyle{mainthmstyle}{}{}{\itshape}{}{\bfseries}{.}{5pt plus 1pt minus 1pt}{\bf Theorem \Alph{mainthmcounter}}
\theoremstyle{mainthmstyle}
\newtheorem{mainthm}[mainthmcounter]{Theorem}

\theoremstyle{definition}
\swapnumbers
\newtheorem{void}[subsection]{}

\title[Arithmetic of characteristic $p$ special $L$-values]
{Arithmetic of characteristic $p$ special $L$-values \\
 (with an appendix by V. Bosser)}

\author{Bruno Angl\`es}
\address{Universit\'e de Caen, CNRS UMR 6139, Campus II, Boulevard Mar\'echal Juin, B.P. 5186, 14032 Caen Cedex, France.}
\email{bruno.angles@unicaen.fr}

\author{Lenny Taelman}
\address{Mathematisch Instituut, P.O. Box 9512, 2300 RA Leiden, The Netherlands}
\email{lenny@math.leidenuniv.nl}

\begin{document}

\begin{abstract}
Recently the second author has associated a finite $\bF_q[T]$-module $H$ to the Carlitz module over a finite extension of $\bF_q(T)$. This module is an analogue of the ideal class group of a number field.

In this paper we study the Galois module structure of this module $H$ for `cyclotomic' extensions of $\bF_q(T)$. We obtain function field analogues of some classical results on cyclotomic number fields, such as the $p$-adic class number formula, and a theorem of Mazur and Wiles about the Fitting ideal of ideal class groups. 
We also relate the Galois module $H$ to Anderson's module of circular units, and give a negative answer to Anderson's Kummer-Vandiver-type conjecture.

These results are based on a kind of equivariant class number formula which refines the second author's class number formula for the Carlitz module.
\end{abstract}

\maketitle

\setcounter{tocdepth}{1}
\tableofcontents

\section{Introduction}

\begin{void}
Let $q$ be a prime power and $A:=\bF_q[T]$ the polynomial ring in one variable $T$ over a finite field $\bF_q$ with $q$ elements. Let $P\in A$ be monic and irreducible. The special $L$-values referred to in the title are values at $s=1$ of $\infty$-adic and $P$-adic Goss $L$-functions associated with various characters of $(A/PA)^\times$.
\end{void}

\begin{void}\label{infL}
Let us first define the relevant $\infty$-adic $L$-values. Let $k_\infty:=\bF_q((T^{-1}))$ be the completion of $k$ at the place at infinity. Let $F$ be a field extension of $\bF_q$ and let $\chi\colon (A/PA)^\times \to F^\times$ be a homomorphism. We define the \emph{$\infty$-adic $L$-value of $\chi$ at $1$} by the series
\begin{equation}\label{eqinfL}
	L(1,\chi) := \sum_{a\in A_+} \frac{\chi(a)}{a} \in F \otimes_{\bF_q} k_\infty,
\end{equation}
where $A_+$ denotes the set of monic elements of $A$, and where we define $\chi(a)$ as follows:
\[
	\chi(a) := \begin{cases} 
	\chi(a+PA) & \text{if } a \not\in PA, \\
	1 & \text{if }a \in PA\text{ and }\chi=1, \\
	0 & \text{if }a \in PA\text{ and }\chi\neq 1. \end{cases}
\]
This series converges because there are only finitely many elements of $A$ of bounded $\infty$-adic valuation.
\end{void}

\begin{void}\label{intPL}
For the $P$-adic $L$-values, consider the completion $A_P := \varprojlim_n A/P^nA$ and a homomorphism $\chi\colon (A/PA)^\times \to A_P^\times$. Then we define the \emph{$P$-adic $L$-value of $\chi$ at $1$} by
the series
\begin{equation}\label{eqPL}
	L_P(1,\chi) := \sum_{n\geq 0} \sum_{a\in A_{n,+}} \frac{\chi(a)}{a} \in A_P,
\end{equation}
where $A_{n,+}\subset A$ is the set of monic elements of degree $n$, and where this time we define $\chi(a)$ by
\[
	\chi(a) := \begin{cases} 
	\chi(a+PA) & \text{if } a \not\in PA, \\
	0 & \text{if }a \in PA. \end{cases}
\]
The convergence of this series is much more subtle. It follows from either \cite[Lemma 3.6.7]{GOS3} or \cite[\S 4.10]{AND} that the series (\ref{eqPL}), with the terms grouped as indicated, converges in $A_P$.
\end{void}

\begin{void}
In this paper we study arithmetic properties of these special $L$-values. In particular, we prove function field versions of various results about cyclotomic number fields such as the theorem of Mazur and Wiles relating the Fitting ideal of class groups to Bernoulli numbers, and the $p$-adic class number formula. We also consider various analogues of the Kummer-Vandiver problem.
\end{void}

\begin{void}
The arithmetic properties encoded by these $L$-values are closely related to the \emph{Carlitz module} (a particular Drinfeld module), and to the ``unit module'' and ``class module'' associated to the Carlitz module by the second author \cite{TAE1,TAE2}. One of the principal objectives of this paper is to relate the Galois module structure of these modules to the above special $L$-values.

In the next section we recall some of the theory of the Carlitz module, and state our main results. Along the way we fix some notation.
\end{void}

\section{Statement of the principal results}

\begin{void}
Let $A:=\bF_q[T]$. For any $A$-algebra $R$ denote by $\C(R)$ the $A$-module whose underlying $\bF_q$-vector space is $R$, equipped with the unique $A$-module structure
\[
	A \times \C(R) \to \C(R)
\]
satisfying
\[
	 (T,r) \mapsto Tr + r^q
\]
for all $r \in R$. The resulting functor $\C$ from the category of $A$-algebras to the category of $A$-modules is called the \emph{Carlitz module}. It is a \emph{Drinfeld module} of rank $1$. See \cite{GOS} for more background on Drinfeld modules and on the Carlitz module.
\end{void}

\begin{void}\label{expdef}
Let $k=\bF_q(T)$ be the fraction field of $A$. There is a unique power series
$\exp_\C X$ of the form
\[
	\exp_\C X = X + e_1 X^q + e_2 X^{q^2} + \cdots \in k[[X]]
\]
such that
\begin{equation}\label{expfuneq}
	\exp_\C (TX) = T\exp_\C X + (\exp_\C X)^q.
\end{equation}
 This power series is called the \emph{Carlitz exponential}.
If $F$ is a finite extension of $k_\infty := \bF_q((T^{-1}))$ then the power series $\exp_\C$ defines an entire function on $F$ and the functional equation (\ref{expfuneq}) implies that $\exp_\C$ defines an $A$-module homomorphism $\exp_\C\colon F \to \C(F)$. 
\end{void}

\begin{void}
Now let $K$ be a finite extension of $k$. Let $\cO_K$ be the integral closure of $A$ in $K$. Define
$K_\infty := K \otimes_k k_\infty$. The $k_\infty$-algebra $K_\infty$ is canonically isomorphic with $\prod_{v\mid \infty} K_v$ (and hence the reader should beware that $K_\infty$ is not necessarily a field, as the notation might suggest).
The Carlitz exponential defines an $A$-linear map
\[
	\exp_\C \colon K_\infty \to \C(K_\infty).
\]
It is shown in \cite{TAE1} that the $A$-module
\[
	\U(\cO_K) := \left\{ \gamma \in K_\infty \colon \exp_\C \gamma \in \C(\cO_K) \right\}
\]
is finitely generated, and that the $A$-module
\[
	\H(\cO_K) := \frac{\C(K_\infty)}{ \C(\cO_K) + \exp_\C K_\infty  }
\]
is finite.  

$\H(\cO_K)$ is an $A$-module analogue of the ideal class group of a number field and $\U(\cO_K)$ is an $A$-module analogue of the lattice of logarithms of units in a number field.

We will denote by $\cU \subset \C(\cO_K)$ the image of $\U(\cO_K)$ under $\exp_\C$. Whereas $\cU$ is a finitely generated $A$-module, the $A$-module $\C(\cO_K)$ is \emph{not} finitely generated by \cite{POO}.
\end{void}

\begin{void}
Let $P\in A$ be monic irreducible and denote its degree by $d$. Let $K$ be the splitting field of the $P$-torsion of the Carlitz module over $k$. In the rest of this paper $K$ will denote this particular finite extension of $k$, associated to the fixed prime $P$.
\end{void}

\begin{void}\label{preelt}
This extension $K/k$ is often called the ``cyclotomic extension'' (associated to the prime $P$). It has been studied extensively and in section \ref{secelt} we review some of its properties. In particular, we will see that $K/k$ is an abelian extension of degree $q^d-1$, whose Galois group $\Delta$ is canonically isomorphic with $(A/PA)^\times$ (through its action on $\C(K)[P]\cong A/PA$). The extension is unramified away from $P$ and $\infty$, it is totally ramified at $P$, and the decomposition and inertia groups at $\infty$ both coincide with the subgroup $\bF_q^\times$ of $(A/PA)^\times$.
\end{void}

 Our first result is a kind of \emph{equivariant class number formula}, relating the special values $L(1,\chi)$ to the $A[\Delta]$-modules $\H(\cO_K)$ and $\U(\cO_K)$. 

\begin{void}
To state the theorem, it is convenient to group all the $L(1,\chi)$ together in one equivariant $L$-value as follows. There is a unique element $L(1,\Delta) \in k_\infty[\Delta]^\times$ with the property that for every field extension $F/\bF_q$ and for every homomorphism $\chi\colon \Delta \to F^\times$ the image of $L(1,\Delta)$ under the homomorphism $k_\infty[\Delta] \to F\otimes k_\infty$ induced by $\chi$ equals $L(1,\chi)$.
\end{void}

\begin{void}
$K_\infty$ is free of rank one as a $k_\infty[\Delta]$-module, and it contains sub-$A[\Delta]$-modules $\cO_K$ and $\U(\cO_K)$. By \cite{TAE1} the natural maps $k_\infty \otimes_A \cO_K \to K_\infty$ and
$k_\infty \otimes_A \U(\cO_K) \to K_\infty$ are isomorphisms. Since $A[\Delta]$ is isomorphic to a finite product of rings $F[T]$ with $F$ a finite extension of $\bF_q$, it is a principal ideal ring. We conclude that $\cO_K$ and $\U(\cO_K)$ are free of rank one as $A[\Delta]$-modules.

\end{void}
\begin{mainthm}\label{thmCNF} We have
\[
	L(1,\Delta) \cdot \cO_K = \Fitt_{A[\Delta]} \H(\cO_K) \cdot \U(\cO_K)
\]
inside $K_\infty$, where $\Fitt_{A[\Delta]} \H(\cO_K)$ is the Fitting ideal of the $A[\Delta]$-module $\H(\cO_K)$.
\end{mainthm}

This is an equivariant refinement of a special case of the class number formula of \cite{TAE2}, and our proof (see section \ref{secCNF}) follows closely the argument of \emph{loc.~cit.}

\bigskip

\begin{void}\label{oddeven}
For our further results we need to split $\H(\cO_K)$ into an ``odd'' and an ``even'' part, which we now define. Note that we have $\bF_q^\times \subset \Delta=(A/PA)^\times$. Let $M$ be an $A[\Delta]$-module. Let
$e^- \in A[\bF_q^\times]$ be the idempotent corresponding to the tautological character $\bF_q^\times \to \bF_q^\times, x \mapsto x$. We define the \emph{odd} part of $M$ as 
\[
	M^- := e^- M
\]
and the \emph{even} part of $M$ as 
\[
	M^+ := (1-e^-) M.
\]
Clearly we have $M = M^+ \oplus M^-$ for every $A[\Delta]$-module $M$. Correspondingly the ring $A[\Delta]$ factors as $A[\Delta]^+ \times A[\Delta]^-$.

The subgroup $\bF_q^\times$ of $\Delta$ is the decomposition group at $\infty$ in $K/k$, and as such it is analogous to the subgroup generated by complex conjugation in the Galois group of a cyclotomic extension of $\bQ$. Our use of the terms ``odd'' and ``even'' is motivated by this analogy.
\end{void}

\begin{void}
Similarly, if $F$ is a field extension of $\bF_q$ and $\chi\colon \Delta \to F^\times$ a homomorphism then we say that $\chi$ is \emph{odd} if $\chi$ restricts to the identity map on $\bF_q^\times \subset \Delta$, and \emph{even} otherwise. If $F$ contains a field of $q^d$ elements and $M$ is an $A[\Delta]$-module then we have a decomposition of $F\otimes_{\bF_q}\! A$-modules
\[
	F\otimes_{\bF_q}\! M = \bigoplus_{\chi\colon \Delta\to F^\times} e_\chi (F\otimes_{\bF_q}\! M).
\]
where $\chi$ ranges over all homomorphisms and where $e_\chi \in (F\otimes A)[\Delta]$ denotes the idempotent associated to $\chi$.  The submodules $F\otimes_{\bF_q} M^+$ and $F\otimes_{\bF_q} M^-$ of $F\otimes_{\bF_q} M$ are obtained by restricting the direct sum to even or odd $\chi$ respectively.
\end{void}

\bigskip

\begin{void}\label{B1s}
We now consider the odd part $\H(\cO_K)^-$. We will give a formula for the Fitting ideal of the $A[\Delta]$-module  $\H(\cO_K)^-$ similar to the theorem of Mazur-Wiles \cite[p.~216, Theorem 2]{MW} relating the $p$-part of the class group of $\bQ({\zeta_p})$ to generalized Bernoulli numbers. However, we give a full description of the Fitting ideal, not only of its $P$-part. 

As stated above, $\cO_K$ is free of rank one as an $A[\Delta]$-module. Let 
$\eta$ be a generator of $\cO_K$ as an $A[\Delta]$-module and let $\lambda \in \cO_K$ be a non-zero $P$-torsion element of $\C(\cO_K)$. Let $F$ be a field containing $\bF_q$ and $\chi\colon \Delta \to F^\times$ a homomorphism. Then there is a unique
$B_{1,\chi} \in F \otimes_{\bF_q} k$ such that
\[
	e_\chi (1\otimes\lambda^{-1}) = B_{1,\chi} e_\chi (1\otimes \eta)
\]
in $F\otimes_{\bF_q} K$. Note that $B_{1,\chi}$ depends on the choice of $\lambda$ and $\eta$, but only up to a scalar in $F^\times$. In \S \ref{secnormalbasis} we will single out for each $\chi$ a particular $B_{1,\chi}$, independent of choices. We will call these \emph{generalized Bernoulli-Carlitz numbers}. They relate to the Bernoulli-Carlitz numbers (see \ref{defBCprime} and \ref{defBCn}) in the same way the generalized Bernoulli numbers relate to the usual Bernoulli numbers.
\end{void}

\begin{mainthm}\label{thmodd}
Let $F$ be a field containing $\bF_q$ and let $\chi\colon \Delta \to F$ be an odd character. Consider the ideal $I = \Fitt e_\chi (F \otimes_{\bF_q} \H(\cO_K))$ in $F\otimes_{\bF_q}\!A$. Then
\begin{enumerate}
\item $I=(1)$ if $\chi=1$ (and then $q=2$);
\item $I=( (1\otimes T-\chi(T)\otimes1) B_{1,\chi^{-1}})$ if $\chi$ extends to a ring homomorphism $A/PA \to F$;
\item $I=(B_{1,\chi^{-1}})$ otherwise.
\end{enumerate}
\end{mainthm}

In the first case ($q=2$ and $\chi=1$), we have $B_{1,\chi^{-1}}=(P+1)/(T^2+T)$, see \ref{B11}.

\begin{void}\label{defBCprime}
For all non-negative integers $n$ we define $\BC'_n \in k$ by the power series identity
\[
	\frac{X}{\exp_\C X} = \sum_{n\geq 0} \BC'_n X^n.
\]
These $\BC'_n$ are (up to a normalisation factor, see also \ref{defBCn}) the \emph{Bernoulli-Carlitz numbers} introduced by Carlitz, who related them to certain Goss zeta values. In \S \ref{secodd} we establish congruences relating the $B_{1,\chi^{-1}}$ and $\BC'_n$ and use these to obtain a new proof of the analogue of the \emph{Herbrand-Ribet theorem} established in \cite{TAE3}:
\begin{mainthm}\label{mainthmHR}
Let $\omega\colon \Delta \to (A/PA)^\times$ be the tautological character. Let $1< n < q^d-1$ be divisible by $q-1$. Then
\[
	e_{\omega^{1-n}} ((A/PA) \otimes_A \H(\cO_K)) \neq 0
\]
if and only if $ v_P(\BC'_n) > 0$.
\end{mainthm}

\end{void}

\bigskip

We have no complete description of the Fitting ideal of the even part $\H(\cO_K)^+$, but we give a kind of $P$-adic class number formula involving the $P$-part of $\H(\cO_K)^+$. To state the theorem we need a $P$-adic version of the module $\cU$.

\begin{void}
By \ref{preelt} there is a unique prime of $K$ above $P\in A$. Let $\cO_{K,P}$ be the completion of $\cO_K$ at this unique prime.
Let $\fm$ be the maximal ideal of $\cO_{K,P}$. For every $N\geq 0$ the subgroup $\fm^N$ is
 stable under the Carlitz $A$-action and we denote the resulting $A$-module by $\C(\fm^N) \subset \C(\cO_{K,P})$.  Now assume that $N\geq 2$. Under this assumption, we show
in Proposition \ref{propPexp} that the $A$-action on $\C(\fm^N)$ extends uniquely to a continuous
$A_P$-module structure on $\C(\fm^N)$, and that the resulting $A_P$-module is torsion-free.

In fact, also  $\C(\fm)$ (but not $\C(\cO_{K,P})$) has a  natural $A_P$-module-structure. However, $\C(\fm)$ is not torsion-free, and since the presence of torsion would slightly complicate some of our statements, we decided to work with $\C(\fm^2)$ everywhere.
\end{void}

\begin{void}
Denote by $\cU_2$ the intersection $\cU \cap \C(\fm^2)$ and by $\widebar{\cU_2}$ its
topological closure inside $\C(\fm^2)$. Then $\widebar{\cU_2}$ is a sub-$A_P[\Delta]$-module of $\C(\fm^2)$, which is free over $A_P$.
\end{void}

\begin{void}\label{Padicidempotents}
The reduction map $A_P \to A/PA$ has a unique section which is a ring homomorphism, giving $A_P$ the structure of an $A/PA$-algebra. In particular, every $A_P[\Delta]$-module $M$ decomposes as
\[
	M = \bigoplus_\chi e_\chi M
\]
where $\chi$ runs over all homomorphisms $\chi\colon \Delta \to A_P^\times$, and where $e_\chi\in A_P[\Delta]$ is the idempotent associated to $\chi$. We call a homomorphism $\chi\colon \Delta \to A_P^\times$ \emph{odd} if its restriction to $\bF_q^\times$ is the inclusion map $\bF_q^\times \to A_P^\times$, and \emph{even} otherwise.
\end{void}

Using a $P$-adic Baker-Brumer theorem of Vincent Bosser (see the appendix) we show the following Leopoldt-type result:

\begin{mainthm}\label{mainthmLeopoldt} 
If $\chi\colon \Delta \to A_P^\times$ is even then $e_\chi(\C(\fm^2)/\widebar{\cU_2})$ is finite.
\end{mainthm}

Our main result regarding the even part of $\H(\cO_K)$ is the following theorem.

\begin{mainthm}\label{thmeven}
Let $\chi\colon \Delta \to A_P^\times$ be even. Then $L_P(1,\chi)\neq 0$ and
\[
	\length_{A_P} e_\chi \left( A_P\otimes_A \H(\cO_K)\right)  +
	\length_{A_P} e_\chi \frac{\C(\fm^2)}{\widebar{\cU_2}} = v_P( L_P(1,\chi) ).
\]
\end{mainthm}

We also show that $L_P(1,\chi)=0$ for odd $\chi$.

\begin{void}
An important ingredient in the proof of Theorem \ref{thmeven} is Anderon's module $\cL$ of \emph{special points} \cite{AND}. This is a finitely generated submodule of $\C(\cO_K)$, constructed through explicit generators.  It is a Carlitz module analogue of the group of circular units (or cyclotomic units) in cyclotomic number fields. We refer to section \ref{cycunits} for the definition. We denote by $\sqrt{\cL}$ its division hull in $\C(\cO_K)$, that is,
\[
	\sqrt{\cL} := \left\{ m \in \C(\cO_K) \colon \text{there is an } a \in A \setminus \{0\} \text{ such that } am \in \cL \right\}.
\]

In \S \ref{cycunits} we will show 
\end{void}

\begin{mainthm}
$\sqrt{\cL}=\cU$, the quotient\, $\cU/\cL$ is finite, and we have
\[
	\Fitt_{A[\Delta]} \cU/\cL = \Fitt_{A[\Delta]} \H(\cO_K)^+.
\]
\end{mainthm}

As in the classical case, we do not expect  $\cU/\cL$ and $\H(\cO_K)^+$ to be isomorphic $A[\Delta]$-modules in general.

\begin{void}
Motivated by the \emph{Kummer-Vandiver conjecture}, Anderson had conjectured \cite[\S 4.12]{AND} that the $P$-torsion of $\sqrt{\cL}/\cL$ is trivial, and we now see that this is equivalent with the statement that the $P$-torsion of $\H(\cO_K)^+$ is trivial. Recently we have constructed examples where the the latter does not hold \cite{ANG&TAE}, and we therefore conclude that also Anderson's conjecture is false. For example:
\begin{mainthm}Let $q=3$ and $P=T^9-T^6-T^4-T^3-T^2+1$ in $\bF_3[T]$. Then $\cU/\cL$ has non-trivial $P$-torsion.
\end{mainthm}

\end{void}

\section{$A[\Delta]$-modules}

\begin{void}
Let $P \in A$ be an irreducible element of degree $d$, and let $\Delta:=(A/PA)^\times$. In this section we collect some elementary facts on the structure of $A[\Delta]$-modules, and fix some notation.

Note that $\bF_q[\Delta] = \prod_i F_i$ for some finite field extensions $F_i/\bF_q$. As a consequence we have $A[\Delta] = \prod_i F_i[T]$. In particular $A[\Delta]$ is a principal ideal ring.
\end{void}

\begin{void}
If $M$ is a finite $A[\Delta]$-module then there are ideals $I_1,\ldots,I_n$ such that
\[
	M \cong  A[\Delta]/I_1 \oplus \cdots \oplus A[\Delta]/I_n.
\]
The \emph{Fitting ideal} of $M$ is the ideal
\[
	\Fitt_{A[\Delta]} M := I_1 \cdots I_n.
\]
\end{void}

\begin{void}\label{monicfitting}
Every ideal $I$ of finite index in $A[\Delta]$ has a unique generator $f$ normalized such that for every $i$ the component $f_i \in F_i[T]$  of $f$ is monic.  If $M$ is a finite $A[\Delta]$-module then we denote by $[M]_{A[\Delta]}$ this normalized generator of $\Fitt_{A[\Delta]} M$.
\end{void}

\begin{void}
Let $F$ be an extension of $\bF_q$ and $\chi\colon \Delta\to F^\times$ a homomorphism. Consider the element
\[
	e_\chi := - \sum_{\sigma \in \Delta} \chi^{-1}(\sigma) \sigma \in F[\Delta]
\]
Then $e_\chi$ is an idempotent and $\sigma e_\chi = \chi(\sigma) e_\chi$ for all $\sigma \in \Delta$.
\end{void}

\begin{void}
Let $F$ be a field containing a field of $q^d$ elements. Then the ring 
$F\otimes_{\bF_q} A[\Delta]$ factors as
\[
	F\otimes_{\bF_q} A[\Delta] = \prod_{\chi\colon \Delta\to F^\times} (F\otimes_{\bF_q} A),
\]
and to each $\chi$ corresponds an idempotent $e_\chi \in F \otimes_{\bF_q} A[\Delta]$.
If $M$ is an $A[\Delta]$-module, then we have a decomposition
\[
	F \otimes_{\bF_q} M = \bigoplus_{\chi}  e_\chi (F\otimes_{\bF_q} M).
\]
\end{void}

\begin{void}
Let $F$ be an extension of $\bF_q$ containing a field of $q^d$ elements and let $\Frob\colon F \to F$ be the $q$-Frobenius $x\mapsto x^q$. Then for an 
\[
	\alpha = \sum \alpha(\chi) e_\chi \in F \otimes_{\bF_q} A[\Delta]
\]
with $\alpha(\chi)\in F\otimes_{\bF_q} A$ for all $\chi$ we have that $\alpha$ lies in $A[\Delta]$ if and only if 
\[
	\alpha(\chi^q) = (\Frob\otimes \id) \alpha(\chi)
\]
for all $\chi$.
\end{void}

\begin{void}
Let $V$ be a $k_\infty[\Delta]$-module which is free of rank one. We call a sub-$A[\Delta]$-module $\Lambda$ in $V$ a \emph{lattice} if it is free of rank one over $A[\Delta]$. If $\Lambda_1$ and $\Lambda_2$ are $A[\Delta]$-lattices in $V$ then there is an $f\in k_\infty[\Delta]$ so that $\Lambda_2 = f \Lambda_1$. Moreover, this $f$ is unique if we normalize it analoguously to \ref{monicfitting}, by demanding that for every $i$ its component $f_i\in F_i((T^{-1}))$ has leading coefficient $1$. We denote this normalized $f$ by $[ \Lambda_1 : \Lambda_2 ]_{A[\Delta]}$.
\end{void}

\section{Elementary properties of the cyclotomic function field $K$}\label{secelt}

Recall that $P$ is a monic irreducible element of $A$ and that $K$ denotes the spliting field of the $P$-torsion of the Carlitz module over $k$.
In this section we collect some elementary facts about the field extension $K/k$, and about the Carlitz module over $K$.

\begin{void} \label{torselt}
(\cite[p. 202--208]{ROS}).
We have $\C[P](K)\cong A/PA$ and the action of $\Delta := \Gal(K/k)$ on $\C[P](K)$ induces
a homomorphism
\[
	\omega\colon \Delta \to (A/PA)^\times.
\]
This map is an isomorphism, which we use to identify $\Delta$ with $(A/PA)^\times$.

The field of constants of $K$ is $\bF_q$. The extension $K/k$ is unramified away from $P$ and $\infty$. For a monic irreducible $f \in A$ which is coprime with $P$ we have that $\omega( \Frob_{(f)} ) = \bar f \in (A/PA)^\times$. The prime $P$ is totally ramified in $K/k$.
\end{void}

\begin{void} \label{Alambda}
(\cite[Prop. 12.9 and 12.7]{ROS}).
Let $\lambda\in K$ be a generator of $\C(K)[P]$. Then $\lambda$ is integral over $A$, so $\lambda \in \cO_K$. We have $\cO_K=A[\lambda]$. Moreover, $\lambda$ is a generator of the unique prime ideal of $\cO_K$ that lies above $(P)$.
\end{void}

\begin{void}\label{piexp}
(\cite[\S 2, 3.2, 3.3 and 9.4]{GOS}).
Let $k_\infty^a$ be an algebraic closure of $k_\infty$. Then the exponential map (see \ref{expdef}) defines a short exact sequence
\[
	0 \longto A \bar{\pi} \longto k_\infty^a \overset{\exp_\C}{\longto} \C(k_\infty^a) \longto 0
\]
with
\[
	\bar\pi := \left(\sqrt[q-1]{-T}\right)^q \prod_{n = 1}^\infty \left( 1 - T^{1-q^n} \right)^{-1}
	\in k_\infty(\sqrt[q-1]{-T})
\]
for a choice of $(q-1)$-st root of $-T$. The field $k_\infty(\bar\pi)$ has degree $q-1$ over $k_\infty$.
\end{void}

\begin{void}\label{pilambda}\label{inftydecomposition}
Consider the element
\[
	\lambda := \exp_\C( \bar\pi / P )
\]
of $k_\infty^a$. It is a generator of $\C(k_\infty^a)[P]$. Let $v$ be a place of $K$ above $\infty$. Then we have $K_v \cong k_\infty(\lambda) = k_\infty( \bar\pi  )$. The Galois group of the Kummer extension $k_\infty(\bar\pi)/k_\infty$ is naturally isomorphic to $\bF_q^\times$, and acts on $\lambda$ via the tautological character $\id\colon \bF_q^\times \to \bF_q^\times$. We conclude that the subgroup $\bF_q^\times \subset \Delta$ is both the inertia group and decomposition group at $\infty$ (see also \cite[Theorem 12.14]{ROS}).
\end{void}

\begin{void}\label{ADeltaLambda}
Let $\Lambda$ be the kernel of $\exp_\C \colon K_\infty \to \C(K_\infty)$. Then by the above we have that
$\Lambda$ is free of rank $(q^d-1)/(q-1)$ over $A$ and that $\Lambda^-=\Lambda$. Also,
we have a short exact sequence
\[
	0 \longto \Lambda \longto \U(\cO_K) \overset{\exp_\C}{\longto} \cU \longto 0.
\]
of $A[\Delta]$-modules. Since $\U(\cO_K)$ is free of rank one as $A[\Delta]$-modules we find that
$\cU^+$ is free of rank one over $A[\Delta]^+$ and $\cU^-$ is a torsion $A$-module.
\end{void}

\begin{void}\label{carlitztorsion}
Finally we compute the torsion module of $\C(K)$. Let $Q \in A$ be a nonzero multiple of $P$ and let
$L$ be the splitting field of $\C[Q]$. Then by \cite[Theorem 12.8]{ROS} $L$ is Galois over $K$, and its Galois group can be identified with the kernel of the reduction map $(A/QA)^\times\to (A/PA)^\times$. Its action on $\C(L)[Q]\cong A/QA$ is the natural one. We conclude that
\[
	\C(\cO_K)_\tors=\C(K)_\tors = \C(K)[Q] \cong A/QA.
\]
where $Q\in A$ is the largest multiple of $P$ so that the reduction map $(A/QA)^\times\to (A/PA)^\times$ is an isomorphism. We have $Q=P$ if $q>2$ and $Q$ is the least common multiple of $P$ and $T(T+1)$ if $q=2$. 
\end{void}

\section{Gauss-Thakur sums and generalized Bernoulli-Carlitz numbers}\label{secnormalbasis}

\begin{void}
Fix a generator $\lambda \in K$ of the $P$-torsion of the Carlitz module. Let $F$ be a field extension of $\bF_q$ and $\chi\colon \Delta \to F^\times$ a homomorphism.

Let $\bar{F}$ be an algebraic closure of $F$ and $\omega_1,\ldots,\omega_d$ the $d$ distinct $\bF_q$-embeddings of the field $A/PA$ in $\bar{F}$. Then  $\chi$ can be uniquely written as
\[
	\chi = \omega_1^{s_1} \cdots \omega_d^{s_d}
\]
with $0\leq s_i \leq q-1$ for all $i$ and not all $s_i$ equal to $q-1$. 
If we order the $\omega$'s so that $\omega_i = \omega_{i-1}^q$ for all $i$ and if $n$ is the unique integer with $0\leq n < q^d-1$ and $\chi = \omega_1^n$ then the $s_i$ are the $q$-adic digits of $n$.

The \emph{Gauss-Thakur sum} \cite{THA} associated with $\chi$ is defined as follows:
\begin{equation}\label{eqGaussThakur}
	\tau(\chi) =  \prod_{i=1}^{d} \left(
		- \sum_{\delta \in \Delta} \omega_i(\delta)^{-1} \otimes \delta(\lambda)
	\right)^{s_i} \in \bar{F} \otimes_{\bF_q} \cO_{K}.
\end{equation}
Note that $\Gal(\bar{F}/F)$ permutes the $\omega_i$, but fixes $\tau(\chi)$, so that
$\tau(\chi) \in F \otimes_{\bF_q} \cO_K$. Also note that we have
\[
	\tau(\chi) = \prod_{i=1}^d \tau(\omega_i)^{s_i}.
\]
The reader should be warned that $\tau(\chi)$ depends on the choice of $\lambda$.
\end{void}

We summarize the basic properties of these Gauss-Thakur sums:

\begin{proposition}\label{tauProp}
Let $F$ be an extension of $\bF_q$ and  $\chi\colon \Delta \to F^\times$ a homomorphism. Then
\begin{enumerate}
\item $\tau(\chi) \in e_\chi (F \otimes_{\bF_q} \cO_K)$;
\item if $\chi \neq 1$ then $\tau(\chi)\tau(\chi^{-1}) = (-1)^d \otimes P$;
\item $\tau(1)=1$.
\end{enumerate}
\end{proposition} 

\begin{proof} See \cite[\S 2]{ANG2}.
\end{proof}

\begin{void}\label{B1chi}
In particular, the proposition tells us that $\tau(\chi)$ is nonzero. Since $e_\chi( F\otimes_{\bF_q} K)$ is free of rank one over $F \otimes_{\bF_q} k$, we find that there is a unique $B_{1,\chi} \in F\otimes_{\bF_q} k$ such that
\[
	e_{\chi} ( 1 \otimes \lambda^{-1}) = B_{1,\chi} \tau(\chi)
\]
in $F\otimes_{\bF_q} K$. Note that $B_{1,\chi}$ is fixed under $\Delta$, and so is independent of the choice of $\lambda$. We will refer to the $B_{1,\chi}$ as \emph{generalized Bernoulli-Carlitz numbers}.
\end{void}

\begin{void}\label{B11}
For the trivial character $\chi=1$  we have
\[
	B_{1,1} = e_\chi ( 1 \otimes \lambda^{-1} ) = -1 \otimes \tr_{K/k} \lambda^{-1}.
\]
Since the group $\bF_q^\times$ acts freely on the set of conjugates of $\lambda^{-1}$, we see that $B_{1,1}=0$ if $q>2$. If $q=2$ then we have
\[
	B_{1,1} = 1\otimes\frac{P+1}{T^2+T}.
\]
This follows from the fact that for any $Q\in \bF_2[T]$ different from zero
the $Q$-torsion of the Carlitz module is defined by a polynomial of the form
\[
	\phi_Q(X) = QX + \frac{Q^2+Q}{T^2+T}X^2 + \cdots + X^{2^{\deg Q}}
\]
in $k[X]$.
\end{void}

Finally one can use the $\tau(\chi)$ to give a normal basis for $\cO_K$:

\begin{theorem}\label{thmnormalbasis}
There is a unique $\eta \in\cO_K$ such that for all $F$ and for all $\chi\colon \Delta \to F^\times$ we have
$e_\chi(1\otimes \eta) = \tau(\chi)$ in $F\otimes_{\bF_q} \cO_K$. Moreover, $\cO_K$ is free and generated by $\eta$ as an $A[\Delta]$-module.
\end{theorem}

\begin{proof}See \cite[Th\'eor\`eme 2.5]{ANG2} or \cite{CHAP}.
\end{proof}

If $F$ contains a field of $q^d$ elements then we have
\[
	1\otimes \eta = \sum_{\chi} \tau(\chi)
\]
in $F\otimes_{\bF_q} \cO_K$, were $\chi$ ranges over all homomorphisms from $\Delta$ to $F^\times$.

\section{$\infty$-adic equivariant class number formula} \label{secCNF}

In this section we prove Theorem \ref{thmCNF}. The proof follows very closely the proof of the special value formula in \cite{TAE2}, and rather than copying the whole proof, we give an overview of the argument, while treating in detail  those parts that are different.

\begin{void}
We start by giving an Euler product formula for the equivariant $L$-value $L(1,\Delta)$.
If $\fm$ is a maximal ideal of $A$ then both $\cO_K/\fm\cO_K$ and $\C(\cO_K/\fm\cO_K)$ are finite $A[\Delta]$-modules so we can consider the normalized generators $[\cO_K/\fm\cO_K]_{A[\Delta]}$ respectively $[\C(\cO_K/\fm\cO_K)]_{A[\Delta]}$ of their Fitting ideals, see \ref{monicfitting}.

Similarly, if $F$ is an extension of $\bF_q$ and $M$ a finite $F\otimes_{\bF_{q}}\! A$-module then we denote by $[M]_{F\otimes_{\bF_q} A}$ the unique monic generator in  $F\otimes_{\bF_q}\! A \cong F[T]$ of the Fitting ideal of $M$.
\end{void}

\begin{proposition}
Let $F$ be an extension of $\bF_q$ and $\chi\colon \Delta \to F^\times$ a homomorphism. Let $\fm \subset A$ be a maximal ideal, with monic generator $f$. Then we have
\[
	\left[ e_\chi( F\otimes_{\bF_q}\! \C(\cO_K/\fm\cO_K) ) \right]_{F\otimes_{\bF_q} A} =
	1 \otimes f - \chi(f) \otimes 1.
\]
\end{proposition}

Here $\chi(f)$ is defined according to the conventions of \ref{infL}.

\begin{proof}
Without loss of generality we may assume that $F$ is algebraically closed. We need to show that
\[
	\det_{F[Z]} \left( Z - T - \tau \,\mid\, e_\chi \left(F \otimes_{\bF_q} \frac{\cO_K}{f\cO_K}\right) [Z] \right) = f(Z) - \chi(f),
\]
where $\tau$ is the $F[Z]$-linear map induced by the map $\cO_K \to \cO_K, x \mapsto x^q$.

Let $n := \deg f$ and let $t_1$, $t_2$, \ldots, $t_n \in F$ be such that
\[
	f(Z) = \prod_{i=1}^n (Z-t_i)
\]
in $F[Z]$, ordered such that for all $i $ we have $t_{i+1} = t_i^q$ (where the indices are taken modulo $n$).

By the Chinese remainder theorem we have an isomorphism of $F$-algebras 
\begin{equation}\label{crt}
	F \otimes_{\bF_q}\!A/fA \isomto F^n
\end{equation}
which maps $1\otimes (T+fA)$ to $(t_1,\ldots, t_n)$.

Since $\cO_K$ is free of rank one over $A[\Delta]$ (see \ref{thmnormalbasis}), we see that the module 
\[
	M :=  e_\chi \left(F \otimes_{\bF_q} \frac{\cO_K}{f\cO_K}\right)
\]
is free of rank one over $F\otimes_{\bF_q} A/fA \cong F^n$. We thus find that 
\[
	M = M_1 \oplus M_2 \oplus \cdots \oplus M_n
\]
as $F$-vector spaces, with $T m = t_i m$ for all $i$ and all $m\in M_i$. Also, 
the $F$-linear action of $\tau$ on 
$M$ permutes the $n$ components cyclically. So with respect to a suitable $F$-basis we find that $T+\tau$ acts on $M$ via the matrix
\[
\left(\begin{array}{ccccc}
t_1 & x_1 &  &   & \\ 
& t_2 & x_2 &  & \\
&  &  t_3 &  x_3 & \\
 &  &  & \ddots & \\
x_n &  & &  & t_n\end{array}\right)
\]
with characteristic polynomial $f(X)-x_1\cdots x_n$. Since $\tau^n$ acts as scalar multiplication by $x_1\cdots x_n$ on $M$ we are reduced to showing that $\tau^n=\chi(f)$ as endomorphisms of $M$.

If $f$ is coprime with $P$ then we have that $\tau^n$ is the reduction of the Frobenius at $f$. By \ref{torselt} this coincides with the action of $\bar{f}\in \Delta$, hence $\tau^n$ acts as $\chi(f)$ on $M$, as desired. 

In the remaining case $f=P$ we have $\cO_K/f\cO_K \cong (A/PA)[\epsilon]/\epsilon^{q^d-1}$ where $\tau^d$ acts as the identity on $A/PA$ and $\tau^d(\epsilon) = 0$. Also in this case, we find that $\tau^n$ acts on $M$ as
\[
	\chi(f)=\begin{cases} 0 & \text{if }\chi\neq 1,\\
	1 & \text{if }\chi = 1. \end{cases}
\]
\end{proof}

\begin{corollary} The infinite product
\[
	\prod_{\fm} \frac{ \big[ \cO_K/\fm\cO_K \big]_{A[\Delta]} }{ \big[ \C(\cO_K/\fm\cO_K) \big]_{A[\Delta]} }
\]
with $\fm$ ranging over the maximal ideals of $A$, converges in $k_\infty[\Delta]$ to $L(1,\Delta)$.
\end{corollary}

\begin{proof}
Let $F$ be an extension of $\bF_q$ of degree $d$. Then we have
\[
	1\otimes L(1,\Delta) = \sum_{\chi\colon \Delta\to F^\times} L(1,\chi)e_\chi
\]
in $F\otimes k_\infty[\Delta]$. Also, we have the Euler product formula
\[
	L(1,\chi) = \prod_{f} \left(1-\frac{\chi(f)}{f} \right)^{-1}
\]
with $f$ running over the monic irreducible elements of $A$. Now by the proposition, the latter equals
\[
	\prod_{\fm} \frac{
		\big[ e_\chi( F \otimes \cO_K/\fm ) ) \big]_{F\otimes A}
	}{
		\big[ e_\chi( F \otimes \C(\cO_K/\fm) ) \big]_{F\otimes A}
	}.
\]
Combining these we find the corollary.
\end{proof}

\begin{void} Next we need a slight generalization of the trace formula of \cite[\S 3]{TAE2}.
 Let $F$ be a finite extension of $\bF_q$. Let $M$ be a free $F\otimes_{\bF_q}\! A$-module of finite rank.  Let $\tau\colon  M \to M$ be an $\bF_q$-linear map such that $\tau((x\otimes a)m)=(x\otimes a^q)\tau(m)$ for all $x\in F$, $a \in A$ and $m\in M$.
 
Let $Z$ be an indeterminate commuting with $\tau$ and let $\Psi$ be a power series
\[
	\Psi = \sum_{i,j\geq 1} a_{ij} \tau^i Z^{-j}
\]
with $a_{ij}\in A$ for all $i,j$, such that for all $j$ there are only finitely many $i$ with $a_{ij} \neq 0$. In other words, the coefficient of $Z^{-j}$ is a polynomial in $\tau$.

Then for every maximal ideal $\fm$ of $A$ there is an obvious $F[[Z^{-1}]]$-linear action of $\Psi$ on $F[[Z^{-1}]] \otimes_F (M/\fm M)$. Also, there is a natural $F[[Z^{-1}]]$-action of $\Psi$ on the $F[[Z^{-1}]]$-module
\[
	F[[Z^{-1}]]  \hat\otimes_F \frac{k_\infty\otimes_A M}{M} 
	:= \left\{ \sum_{i\geq 0} m_i Z^{-i} \colon m_i \in \frac{k_\infty\otimes_A M}{M} \right\}
\]
This endomorphism is \emph{nuclear} in the sense of \cite[\S 2]{TAE2}, so we can take the determinant of $1+\Psi$ acting on this compact module.
\end{void}

\begin{proposition}
The infinite product
\[
	\prod_{\fm} \det_{F[[Z^{-1}]]} \left( 1 + \Psi \,|\,  F[[Z^{-1}]] \otimes_{F}  \frac{M}{\fm M}\right)^{-1},
\]
where $\fm$ runs over the maximal ideals of $A$, converges to 
\[
	\det_{F[[Z^{-1}]]} \left( 1 + \Psi \,|\, F[[Z^{-1}]]  \hat\otimes_F \frac{k_\infty\otimes_A M}{M} \right).
\]
\end{proposition}

\begin{proof}
The only difference with the formula of \cite[\S 3]{TAE2} is that we deal with a $q$-Frobenius but with $F$-linear determinants for various finite extensions $F/\bF_q$.  However, the proof of this generalization is identical to the proof in \cite{TAE2}.
\end{proof}

Put
\[
	\Theta = \frac{1-(T+\tau)Z^{-1}}{1-TZ^{-1}} - 1 = - \sum_{n=1}^{\infty} \tau T^{n-1} Z^{-n}.
\]
Applying the proposition with $\Psi=\Theta$ and $M=e_\chi(F \otimes_{\bF_q} \cO_K)$ for every
$\chi\colon \Delta\to F^\times$ we get:

\begin{proposition}\label{proptrace}We have 
\[
	L(1,\Delta) = 
	\det_{\bF_q[\Delta][[Z^{-1}]]} 
	\left( 
		1 + \Theta \mid  \bF_q[[Z^{-1}]]  \hat\otimes_{\bF_q}\frac{K_\infty}{\cO_K}
	 \right)_{\big| Z=T}
\]
 in $k_\infty[\Delta]=\bF_q[\Delta]((T^{-1}))$. \qed
\end{proposition}

We can now apply the same reasoning as in section \S 5 of \cite{TAE2}:

\begin{proof}[Proof of Theorem \ref{thmCNF}]
The exponential map induces a short exact sequence of compact $A[\Delta]$-modules 
\[
	0 \to \frac{K_\infty}{\U(\cO_K)} \overset{\exp}{\to} \frac{\C(K_\infty)}{\C(\cO_K)} \to \H(\cO_K) \to 0.
\]
The $A$-module $K_\infty/\U(\cO_K)$ is $A$-divisible, and hence a fortiriori $A[\Delta]$-divisible. Since $A[\Delta]$ is a principal ideal ring the above sequence of $A[\Delta]$-modules splits.
After the choice of a splitting, we obtain an isomorphism
\[
	\gamma\colon \frac{K_\infty}{\U(\cO_K)} \times \H(\cO_K) \to
	\frac{\C(K_\infty)}{\C(\cO_K)}.
\]
Since the map $\exp$ is infinitely tangent to the identity (in the sense of \S 4 of \cite{TAE2}), and since
\[
	1 + \Theta = \frac{ 1 - \gamma T \gamma^{-1} Z^{-1} }{ 1 - T Z^{-1} }
\]
we conclude using \cite[Theorem 4]{TAE2} that
\[
	\det_{\bF_q[\Delta][[Z^{-1}]]} \left( 
		1 + \Theta \mid  \bF_q[[Z^{-1}]]  \hat\otimes_{\bF_q}\frac{K_\infty}{\cO_K}
	 \right)_{\big| Z=T}
	=
	\big[ \H(\cO_K) \big]_{A[\Delta]} \big[ \cO_K : \U(\cO_K) \big]_{A[\Delta]}
\]
which proves the theorem.
\end{proof}

\section{Cyclotomic units}\label{cycunits}
This section is based on Anderson's fundamental paper \cite{AND}, in which he explicitly constructed a finitely generated submodule of $\C(\cO_K)$ and related it to the special values $L(1,\chi)$ and $L_P(1,\chi)$. We bypass some of Anderson's proofs by using the equivariant class number formula of the preceding section. 

\begin{void}
Let $\lambda\in K$ be a generator of the $P$-torsion of the Carlitz module. For all $m \geq 0$ define
\begin{equation}\label{defLm}
	\fL_m := \sum_{\sigma \in \Delta} \,\sigma(\lambda)^m \,\otimes
	\left( \sum_{a\in A_{+,\sigma}} \frac{1}{a} \right)
	\,\in K_\infty = K \otimes_{k} k_\infty
\end{equation}
where $A_{+,\sigma}$ is the set of monic elements of $A$ that reduce to $\sigma$ in $A/PA$.
Let $\fM \subset K_\infty$ be the $A$-module generated by all the $\fL_m$. 
\end{void}

\begin{proposition}\label{propdeltamodule}
For all $\sigma\in \Delta$ we have $\sigma \fM = \fM$.
\end{proposition}

\begin{proof}
Let $\sigma \in \Delta$ and $m\geq 0$. We need to show that $\sigma \fL_m \in \fM$.
We have $\sigma(\lambda^m) \in \cO_K=A[\lambda]$, hence
there are $a_i \in A$ so that $\sigma(\lambda^m) = \sum a_i \lambda^i$.
But then we have $\sigma(\fL_m) = \sum a_i \fL_{i} \in \fM$, as desired.
\end{proof}

\begin{corollary}
The submodule $\fM$ of $K_\infty$ is independent of the choice of $\lambda$.
\end{corollary}

\begin{proposition}\label{propcircL}
Let $F$ be a field extension of $\bF_q$ and $\chi\colon \Delta \to F^\times$
a homomorphism. Then we have
\[
	e_\chi( F\otimes_{\bF_q} \!\fM ) = L(1,\chi) \cdot e_\chi( F\otimes_{\bF_q} \!\cO_K )
\]
as sub-$F\otimes_{\bF_q}\! A$-modules of $e_\chi(F\otimes_{\bF_q}\! K_\infty)$.
\end{proposition}

\begin{proof}
For $\sigma\in \Delta$ we have
\[
	e_\chi \sigma(\lambda)^m = \chi(\sigma) e_\chi\lambda^m 
\]
hence
\[
	e_\chi (1\otimes \fL_m) =
	\left( \sum_{\sigma \in \Delta} \sum_{a\in A_{+,\sigma}}
		\chi(\sigma)\otimes\frac{1}{a} 
	\right) e_\chi\lambda^m
	=  L(1,\chi) e_\chi (1\otimes\lambda^m).
\]
In particular, we have that $e_\chi(F\otimes_{\bF_q}\fM)$ is generated by
\begin{equation}\label{genset1}
	\{ L(1,\chi) e_\chi (1\otimes\lambda^m) \colon m \geq 0 \} \subset e_\chi(F\otimes_{\bF_q} K_\infty)
\end{equation}
as an $F \otimes_{\bF_q} A$-module. Because $\cO_K=A[\lambda]$ (see \ref{Alambda}) we also have that
$e_\chi(F\otimes_{\bF_q} \cO_K)$ is generated by
\begin{equation}\label{genset2}
	\{ e_\chi (1\otimes \lambda^m) \colon m \geq 0 \} \subset e_\chi(F\otimes_{\bF_q} K_\infty)
\end{equation}
as an $F \otimes_{\bF_q}\!A$-module. Comparing the generating sets (\ref{genset1}) and (\ref{genset2}) we obtain the proposition.
\end{proof}

Assembling the isotypical components together we obtain

\begin{theorem}\label{CNF2}
$\fM = L(1,\Delta) \cdot \cO_K$
as  $A[\Delta]$-submodules of $K_\infty$. \qed
\end{theorem}

In particular we have:

\begin{corollary}$\fM$ is free of rank one over $A[\Delta]$.\qed
\end{corollary}

Comparing Theorems \ref{thmCNF} and \ref{CNF2} leads to:

\begin{corollary}\label{compareFitting}$\fM \subset \U(\cO_K)$ and
the $A[\Delta]$-modules $\H(\cO_K)$ and $\U(\cO_K)/\fM$ have the same Fitting ideal.\qed
\end{corollary}

Finally we see that exponentiating the generators of $\fM$ indeed yields integral points on the Carlitz module:

\begin{corollary}
$\exp_{\C} \fM \subset \cU$.\qed
\end{corollary}

\section{Odd part of $\H(\cO_K)$}\label{secodd}

\begin{void}\label{lambdatov}
Fix a place $v$ of $K$ above $\infty$ and an embedding of $K_v$ in $k_v^\sep$. Let $\bar\pi \in K_v$ be the resulting generator of the kernel of
$\exp_\C \colon K_v \to \C(K_v)$, as in \ref{piexp}. Put $\lambda := \exp_\C (\bar\pi/P)$. Then $\lambda$ lies in $K\subset K_v$ and is a generator of $C(K)[P]$. With this choice of $\lambda$, we have a $\tau(\chi)\in F\otimes_{\bF_q} \cO_K$ for every extension $F/\bF_q$ and homomorphism $\chi\colon \Delta \to F^\times$. Recall 
from \ref{B1chi} that $B_{1,\chi} \in F\otimes_{\bF_q} k$ is defined by the relation
\[
	e_{\chi}  ( 1 \otimes \lambda^{-1} ) = B_{1,\chi} \tau(\chi)
\]
in $F\otimes_{\bF_q} K$.
\end{void}

\begin{proposition}\label{B1chiformula}
Let $F$ be a field containing $\bF_q$ and let $\chi\colon \Delta \to F^\times$
be odd.  If $\chi \neq 1$ then 
\[
	L(1,\chi) = \left(1\otimes \frac{\bar\pi}{P} \right) B_{1,\chi^{-1}} \tau(\chi^{-1}) 
\]
in $F\otimes_{\bF_q}\! K_v$. If $\chi=1$ then $q=2$ and
\[
	L(1,\chi)= 1\otimes \frac{ \bar\pi }{ T^2 + T } 
\] 
in $F\otimes_{\bF_q}\! k_\infty$.
\end{proposition}

If $\chi$ extends to a ring homomorphism $A/PA \to F$ then a similar formula for $L(1,\chi)$ has been obtained by Pellarin \cite[Corollary 2]{PEL}.

\begin{proof}[Proof of Proposition \ref{B1chiformula}]
Take the logarithmic derivative of both sides in the product expansion
\[
	\exp_C X = X \prod_{a\in A\backslash \{0\}} \left( 1 - \frac{X}{a\bar\pi} \right)
\]
in $K_v[[X]]$ to find
\[
	\frac{1}{\exp_C X} = \frac{1}{X} + \sum_{a \in A\backslash \{0\}} \frac{1}{X-a\bar\pi} =
	\sum_{a\in A} \frac{1}{X + a\bar\pi}.
\]
Let $b\in A$ be coprime with $P$ and denote by $\sigma_b$ its image in $\Delta$. Substituting $X=\frac{b}{P}\bar\pi$ we obtain 
\begin{equation}\label{lambdaL}
	\frac{1}{\sigma_b(\lambda)} =
	\sum_{a\in A} \frac{1}{(a+\frac{b}{P})\bar\pi} = 
	\frac{P}{\bar\pi} \sum_{a \in b + PA } \frac{1}{a}.
\end{equation}

Now assume $\chi \neq 1$. Tensoring both sides in (\ref{lambdaL}) with $\chi(b)$ and summing over all classes of $b$ in $\Delta=(A/PA)^\times$ we find
\[
	e_{\chi^{-1}}\left(1\otimes \frac{1}{\lambda}\right) = 
	- \left(1\otimes\frac{P}{\bar\pi}\right) \sum_{a\in A} 
	\chi(a)\otimes\frac{1}{a}
\]
in $F\otimes_{\bF_q}\! K_v$. Since $\chi$ is odd we have
\[
	-\sum_{a\in A} \chi(a) \otimes \frac{1}{a} = \sum_{a\in A_+} \chi(a) \otimes \frac{1}{a},
\]
so we find
\[
	e_{\chi^{-1}} \left( 1 \otimes \lambda^{-1} \right)
	= \left(1\otimes \frac{P}{\bar\pi} \right) L(1,\chi).
\]
By \ref{B1chi} we conclude
\[
	B_{1,\chi^{-1}} \tau(\chi^{-1}) = L(1,\chi) \left(1\otimes\frac{P}{\bar\pi}\right)
\]
in $F\otimes_{\bF_q} K_v$, what we had to prove.

For the case where $\chi=1$ and $q=2$ we sum (\ref{lambdaL}) over all $b$ to get
\[
	\tr_{K/k} \frac{1}{\lambda} = \frac{P}{\bar\pi} \sum_{a \in A \setminus PA } \frac{1}{a}
	=
	\frac{P-1}{\bar\pi}  \sum_{a\in A_+} \frac{1}{a}.
\]
 Using \ref{B11} we conclude
$L(1,\chi)=1\otimes \bar\pi/(T^2+T)$, as claimed.
 \end{proof}

\begin{void}\label{piv}
Let $v$ and $\bar\pi \in K_v$ be as in \ref{lambdatov}. Let 
\[
	\bar\pi_v = (0,\ldots, 0, \bar\pi, 0, \ldots 0) \in K_\infty
\]
be the element of $K_\infty$ that projects to $\bar\pi \in K_v$ and to $0$ in $K_w$ for $w\mid \infty$ with $w\neq v$.
\end{void}

\begin{proposition}\label{Lambdagen}
$\Lambda$ is a free rank one $A[\Delta]^-$-module, generated by $\bar\pi_v$.
\end{proposition}

\begin{proof}
Clearly $\Lambda$ is generated by $\{ \sigma(\bar\pi_v) \colon \sigma \in \Delta \}$ as an $A$-module, and since $\Lambda^+=0$ (see \ref{ADeltaLambda}) we find that $\Lambda = A[\Delta]^- \bar\pi_v$. Both $\Lambda$ and $A[\Delta]^-$ are free of rank $(q^d-1)/(q-1)$ over $A$ so we conclude that $\Lambda$ is the free $A[\Delta]^-$-module generated by $\bar\pi_v$.
\end{proof}

\begin{proposition}\label{Bbar} If $\chi\colon \Delta \to F^\times$ is odd and $\chi\neq 1$ then
\[
	L(1,\chi) e_\chi (F\otimes_{\bF_q} \cO_K) = B_{1,\chi^{-1}} e_\chi (F\otimes_{\bF_q} \Lambda)
\]
in $F\otimes_{\bF_q}\! K_\infty$.
\end{proposition}

\begin{proof}
Both sides are free $F\otimes_{\bF_q}\! A$-modules of rank one. By Theorem \ref{thmnormalbasis} the left-hand-side is generated by
\[
	 L(1,\chi)\tau(\chi) \in F\otimes_{\bF_q}\!K_\infty
\]
and by Proposition \ref{Lambdagen} the right-hand-side is generated by
\[
	B_{1,\chi^{-1}}e_\chi (1\otimes\bar\pi_v) \in F\otimes_{\bF_q}\!K_\infty.
\]

Let $\alpha$ be the quotient of these generators:
\begin{equation}\label{defalpha}
	\alpha := \frac{B_{1,\chi^{-1}}e_\chi(1\otimes\bar\pi_v)}{L(1,\chi)\tau(\chi)} \in (F\otimes_{\bF_q}\!K_\infty)^\times.
\end{equation}
We need to show $\alpha  = x\otimes 1$ for some $x\in F^\times$. Since $\Delta$ acts via $\chi$ on both the numerator and the denominator of (\ref{defalpha}), we have that $\alpha$ is invariant under $\Delta$. It therefore suffices to show that the $v$-component
$\alpha_v \in F \otimes_{\bF_q} \! K_v$ is of the form $x\otimes 1$ for some $x \in F^\times$.

Recall that $\bF_q^\times \subset \Delta$ is the decomposition group at $\infty$, and that it acts on $\bar\pi \in K_v$ through the tautological character. In particular, for $\delta \in \Delta$ the $v$-component of
$\delta(\bar\pi_v)$ equals $0$ if $\delta \not\in \bF_q^\times$ and $\delta \cdot \bar\pi$ if $\delta \in \bF_q^\times$ (where the dot denotes field multiplication in $K_v$). Since $\chi$ is odd this gives
\[
	\alpha_v = 
	\frac{ B_{1,\chi^{-1}} (1\otimes\bar\pi) }
	{
	L(1,\chi)\tau(\chi)  
	} \in F\otimes_{\bF_q}\!K_v,
\]
and by Proposition \ref{B1chiformula}
\[
	\alpha_v = \frac{ 1\otimes P }{ \tau(\chi^{-1})\tau(\chi) }.
\]
Using Proposition \ref{tauProp} we conclude $\alpha_v=(-1)^d$, which finishes the proof.
\end{proof}

\begin{lemma}\label{lemmaminustorsion}
$\cU^-=\cU_{\tors} = \C(\cO_K)_\tors$.
\end{lemma}

\begin{proof}
By \ref{ADeltaLambda} we obtain a short exact sequence of $A[\Delta]^-$-modules
\[
	0 \to \Lambda \to \U(\cO_K)^- \to \cU^- \to 0
\]
and since $\U(\cO_K)$ is free of rank one over $A[\Delta]$, we find that $\Lambda$ and $\U(\cO_K)^-$ have the same $A$-rank. We conclude that $\cU^-$ is torsion. Since $\Lambda^+=0$, the module $\cU^+$ is torsion-free, so $\cU^-=\cU_\tors$. In \cite[Prop. 2]{TAE1} it is shown that $\C(\cO_K)_\tors \subset \cU$, so we conclude $\cU_\tors=\C(\cO_K)_\tors$.
\end{proof}

We can now prove Theorem \ref{thmodd}:

\begin{theorem} \label{HB}
Let $F$ be a field containing $\bF_q$ and let $\chi\colon \Delta \to F$ be an odd character. Consider the ideal $I := \Fitt e_\chi (F \otimes_{\bF_q} \H(\cO_K))$ in $F\otimes_{\bF_q} A$. Then
\begin{enumerate}
\item $I=(1)$ if $\chi=1$ (and then $q=2$);
\item $I=( (1 \otimes T-\chi(T)\otimes 1) B_{1,\chi^{-1}})$ if $\chi$ extends to a ring homomorphism $A/PA \to F$;
\item $I=(B_{1,\chi^{-1}})$ otherwise.
\end{enumerate}
\end{theorem}

\begin{proof}
Let $S$ denote the set of  $\chi\colon \Delta \to F^\times$ that extend to a ring homomorphism $A/PA \to F$. 

The equivariant class number formula (Theorem \ref{thmCNF}) says that
\begin{equation}\label{CNFodd}
	L(1,\chi) \tau(\chi) ( F\otimes_{\bF_q}\! A )= I \cdot e_\chi( F \otimes_{\bF_q}\! \U(\cO_K) )
\end{equation}
in $F\otimes_{\bF_q} K_\infty$.  The preceding lemma gives us a short exact sequence
\[
	0 \to \Lambda \to \U(\cO_K)^- \to \C(K)_\tors \to 0.
\]
If $q>2$ then by \ref{carlitztorsion} we have $\C(K)_\tors \cong A/PA$ on which $\Delta$ acts through the tautological character. From this we get
\[
	e_\chi( F \otimes_{\bF_q}\! \U(\cO_K) ))
	= \begin{cases}
		(1\otimes T-\chi(T)\otimes 1)^{-1} e_\chi( F\otimes_{\bF_q}\! \Lambda )  & \text{if }\chi \in S, \\
		e_\chi( F\otimes_{\bF_q}\! \Lambda )  & \text{otherwise,} 
	\end{cases}
\]
and the Theorem follows from (\ref{CNFodd}) and Proposition \ref{Bbar}.

If $q=2$ and $\chi \neq 1$ then $P$ must have degree at least $2$. By \ref{carlitztorsion} we have $\C(K)_\tors \cong A/PA \times A/(T^2+T)A$ with $\Delta$ acting via the tautological character on the first component, and trivially on the second. Hence also in this case we have
\[
	e_\chi( F \otimes_{\bF_q}\! \U(\cO_K) ))
	= \begin{cases}
		(1\otimes T-\chi(T)\otimes 1)^{-1} e_\chi( F\otimes_{\bF_q}\! \Lambda )  & \text{if }\chi \in S, \\
		e_\chi( F\otimes_{\bF_q}\! \Lambda )  & \text{otherwise,} 
	\end{cases}
\]
and the Theorem follows from (\ref{CNFodd}) and Proposition \ref{Bbar}.

Finally, we consider the case where $q=2$ and $\chi=1$. We have that the module of $\Delta$-invariant elements of $\C(K)_\tors$ is isomorphic with $A/(T^2+T)A$, and hence
\[
	e_\chi( F \otimes_{\bF_q}\! \U(\cO_K) ) = (1\otimes(T^2+T)^{-1}) e_\chi( F\otimes_{\bF_q}\! \Lambda ).
\]
Since $e_\chi( F\otimes_{\bF_q}\! \Lambda ) = (1\otimes \bar\pi) (F\otimes_{\bF_q} A)$ we find
using (\ref{CNFodd}) the identity 
\[
	L(1,\chi)(F\otimes_{\bF_q}\! A) = \left( 1 \otimes (T^{2}+T)^{-1} \bar\pi \right) \cdot I
\]
in $F\otimes_{\bF_q}\! k_\infty$. The theorem now follows from Proposition \ref{B1chiformula}.
\end{proof}

\bigskip

\begin{void}
Next we will prove congruences modulo $P$ between the generalized Bernoulli-Carlitz numbers $B_{1,\chi}$ and the Bernoulli-Carlitz numbers $\BC_n$ (to be defined shortly). We then use these  congruences to give a new proof of the analogue of the Herbrand-Ribet theorem of \cite{TAE3}, based on Theorem \ref{HB}.
\end{void}

\begin{void}
For all $i\geq 0$ let
\[
	D_i := \prod_{j=0}^{i-1} (T^{q^i}-T^{q^j}).
\]
Equivalently (see \cite[\S 3]{GOS}) they can be defined by the identity
\[
	\exp_\C X = \sum_{i \geq 0} \frac{ X^{q^i} }{ D_i }
\]
in $k[[X]]$. 
\end{void}

\begin{void}
Let $n$ be a non-negative integer with $q$-adic expansion
\[
	n = n_0 + n_1 q + n_2 q^2 + \cdots, \quad 0 \leq n_i < q.
\]
The $n$-th \emph{Carlitz factorial} $\Pi(n)$ is defined to be
\[
	\Pi(n) := \prod_{i\geq 0} D_i^{n_i} \in A.
\]
Note that $\Pi(n)$ is not divisible by $P$ for all $n < q^d$.

\end{void}

\begin{void}\label{defBCn}
For all $n \geq 0$ the \emph{Bernoulli-Carlitz numbers} $\BC_n \in k$ are defined by the power series identity
\begin{equation}\label{eqBCn}
	\frac{X}{\exp_\C X} = \sum_{n\geq 0} \BC_n \frac{X^n}{\Pi(n)} \in k[[X]].
\end{equation}
These were introduced by Carlitz \cite{CAR}. Since the left-hand-side of (\ref{defBCn}) is invariant under $X \mapsto \mu X$ for $\mu \in \bF_q^\times$ we see that $\BC_n=0$ if $n$ is not divisible by $q-1$.

Note that $\BC_n = \BC_n' \Pi(n)$, see \ref{defBCprime}.
\end{void}

Fix a generator $\lambda \in \cO_K$ of the $P$-torsion of the Carlitz module.

\begin{void}
\emph{Convention.} The completions $k_P$ and $\cO_{K,P}$ are naturally $A/PA$-algebras. For a $\chi\colon \Delta \to (A/PA)^\times$ we will, by abuse of notation, denote by $B_{1,\chi}$ and $\tau(\chi)$ the images of $B_{1,\chi}$ and $\tau(\chi)$ under the natural maps
\[
	(A/PA)\otimes_{\bF_q} k \to k_P
\]
and
\[
	(A/PA)\otimes_{\bF_q} \cO_K \to \cO_{K,P}
\]
respectively. 
\end{void}

Recall that $\omega$ denotes the tautological character $\Delta\to (A/PA)^\times$.

\begin{proposition}\label{taucong}
Let $n$ be an integer with $0\leq n < q^d-1$. Then in $\cO_{K,P}$ we have the congruence
\[
	\tau(\omega^{n})\equiv \frac{\lambda^n}{\Pi(n)} \pmod {\fm^{n+1}}.
\]
\end{proposition}

\begin{proof}
Writing $n$ in its $q$-adic expansion, we see from the definitions of $\tau(\omega^n)$ and $\Pi(n)$ that it suffices to prove
\[
	\tau(\omega^{q^i}) \equiv \frac{\lambda^{q^i}}{D_i} \pmod {\fm^{q^i+1}}
\]
for all $i$ satisfying $0\leq i < d$. This is shown in \cite[Theorem VI]{THA}. Note that Thakur's notation is different from ours. His $\lambda$ is congruent to our $\lambda$ modulo $\fm^2$, but not necessarily the same (see \cite[Lemma II]{THA}). Also note that there is a typo in the proof of Theorem VI of \emph{loc.~cit.}: the left-hand side of the displayed formula should be $g_j/\lambda^{q^j}$ rather than $g_j/\lambda^{q^h}\!$.
\end{proof}

\begin{theorem}\label{cong}If $n \neq 1$ and $0<n\leq q^d-1$ then $B_{1,\omega^{-n}} \in A_P$ and
\[
	B_{1,\omega^{-n}} \equiv \frac{\Pi(q^d-1-n)}{\Pi(q^d-n)} \BC_{q^d-n}
\]	
modulo $P$.
\end{theorem}

\begin{proof} (Compare with \S 8 of \cite{TAE3}.)
Put $e_i := 1/D_i$, so that the Carlitz exponential is given by the power
series
\[
	\exp_\C X = X + e_1 X^q + \cdots \in k[[X]].
\]
Recall that $\fm$ denotes the maximal ideal of $\cO_{K,P}$ and that $\cO_{K,P}/\fm = A/PA$.
Since the coefficients $e_1, \cdots ,e_{d-1}$ are $P$-integral, the following
``truncated exponential map'' is well-defined: 
\[
	\overline\exp_\C \colon \fm/\fm^{q^d} \to \fm/\fm^{q^d}\!,\,\,
	x \mapsto x + e_1 x^q + \cdots + e_{d-1} x^{q^{d-1}}.
\]
It is $\Delta$-equivariant, and it is an isomorphism since it induces the identity map on the intermediate quotients $\fm^i/\fm^{i+1}$. Expanding and truncating the functional equation (\ref{expfuneq})
we find that $\overline\exp_\C$ defines an isomorphism of $A[\Delta]$-modules
\[
	\overline\exp_\C\colon \fm/\fm^{q^d} \isomto \C(\fm)/\C(\fm^{q^d}).
\]

Let $\bar\beta\in \fm/\fm^{q^d}$ be the unique element such that
\[
	\overline\exp_{\C} \bar\beta  = \lambda + \fm^{q^d}.
\]
For a character $\chi\colon \Delta \to (A/PA)^\times$ we denote by $e_\chi \in A_P[\Delta]$ the corresponding idempotent (as in \ref{Padicidempotents}). Since $\overline{\exp}_\C$ is $\Delta$-equivariant we have for $n$, $m$ in $\{0,\ldots, q^d-2\}$ the identity
\begin{equation}\label{betadecomp}
	e_{\omega^{n}} \bar\beta^m = 
	\begin{cases}
		\bar\beta^m & \text{if $n=m$} \\
		0 & \text{if $n\neq m$}.
	\end{cases}
\end{equation}

Let $\beta \in \fm$ be a lift of $\bar\beta$. Then expanding and truncating (\ref{eqBCn})
we obtain the congruence 
\begin{equation}\label{eqcongruence}
	\frac{1}{\lambda} \equiv
	\sum_{n=0}^{q^d-1} \frac{\BC_n}{\Pi(n)} \beta^{n-1} \pmod{ \fm^{q^d-1} }.
\end{equation}

Moreover, by (\ref{betadecomp}) we have for $n,m\in \{0,\ldots, q^d-2\}$
\[
	e_{\omega^n} \beta^m \equiv 0 \pmod{\fm^{q^d}} \quad \text{ if $n \neq m$ }
\]
and
\[
	e_{\omega^n} \beta^m \equiv \beta^m \pmod{\fm^{q^d}} \quad \text{ if $n = m$}.
\]

Applying $e_{\omega^{-n}} = e_{\omega^{q^d-1-n}}$ to (\ref{eqcongruence}) we obtain
\[
	e_{\omega^{-n}} \lambda^{-1} \equiv \frac{\BC_{q^d-n}}{\Pi(q^d-n)} \beta^{q^d-1-n} \pmod {\fm^{q^d-1}}
\]
and therefore 
\[
	\tau(\omega^{q^{d}-1-n}) B_{1,\omega^{-n}}  \equiv  \frac{\BC_{q^d-n}}{\Pi(q^d-1-n)} \beta^{q^d-1-n} \pmod {\fm^{q^d-1}}.
\]
By Proposition \ref{taucong} and the fact that $\beta/\lambda \equiv 1$ mod $\fm$ this implies
\[
	B_{1,\omega^{-n}} 
	 \equiv
	\frac{\Pi(q^d-1-n)}{\Pi(q^d-n)} \BC_{q^d-n} \pmod{ \fm }.
\]
Since both sides lie in $A_P$, we conclude that the claimed congruence mod $P$
indeed holds.
\end{proof}

If we combine Theorem \ref{HB} with the congruence of Theorem \ref{cong} we obtain a new proof of the Herbrand-Ribet theorem of \cite{TAE3}:

\begin{theorem}\label{thmHR}
Let $1 \leq n < q^d-1$ be divisible by $q-1$. Then
\[
	e_{\omega^{1-n}} (A/PA\otimes_{\bF_q} \H(\cO_K))
\]
is non-zero if and only if $v_P(\BC_n) > 0$. 
\end{theorem}

For $n<q^d-1$ we have $v_P(\BC_n) = v_P(\BC'_n)$ with $\BC'_n$ as in \ref{defBCprime} so that this theorem is equivalent with Theorem \ref{mainthmHR}.

\begin{proof}[Proof of Theorem \ref{thmHR}]
Passing from $A$ to $A_P$ in Theorem \ref{HB} and splitting out character by character we find
\begin{equation}\label{pfHR1}
	\length_{A_P} e_{\chi} (A_P\otimes_A \H(\cO_K)) =
	\begin{cases}
	0 & \text{if $\chi = 1$ (and $q=2$),} \\
	v_P(B_{1,\chi^{-1}}) + 1 & \text{if $\chi = \omega$,} \\
	v_P(B_{1,\chi^{-1}}) & \text{otherwise.}
	\end{cases}
\end{equation}
for all odd $\chi\colon \Delta \to A_P^\times$.

If $n=1$ and $q=2$ then $\BC_1 = (T^2+T)^{-1}$. Since we must have $d\geq 2$ we find $v_P(\BC_1)=0$ and the Theorem holds.

If $n>1$ then we see that $e_{\omega^{1-n}} (A/PA \otimes_{\bF_q} \H(\cO_K))$ is non-zero if and only if
$v_P(B_{1,\omega^{n-1}})>0$, and by Theorem \ref{cong} this is the case if and only if
$v_P(\BC_n)>0$.
\end{proof}

\section{Even part of $\H(\cO_K)$}\label{seceven}

%
%
%
%

\begin{void}Let $\cL\subset \C(\cO_K)$ be the image of $\fM$ in $\C(\cO_K)$ and $\sqrt{\cL}$ its division hull in $\C(\cO_K)$, that is,
\[
	\sqrt{\cL} := \left\{ m \in \C(\cO_K) \colon \exists a \in A \setminus \{0\} \text{ such that } am \in \cL \right\}.
\]
\end{void}

\begin{proposition}\label{propdivclosure}$\sqrt{\cL}=\cU$.
\end{proposition}

\begin{proof}See the remark after \cite[Prop. 2]{TAE1}.
\end{proof}

\begin{theorem} \label{thmdivclosure}
$\Fitt_{A[\Delta]}  \cU/\cL = \Fitt_{A[\Delta]} \H(\cO_K)^{+}$. 
\end{theorem}

\begin{proof}
By \ref{ADeltaLambda} and Lemma \ref{lemmaminustorsion}
the minus-part of the short exact sequence of $A[\Delta]$-modules
\[
	0 \to \ker \exp_C \to \U(\cO_K) \to \cU \to 0
\]
is the short exact sequence
\[
	0 \to \ker \exp_C \to \U(\cO_K)^- \to \C(K)_\tors \to 0.
\]
By Proposition \ref{propdivclosure} we have that $\cL_\tors = \cU_\tors$ so that the minus-part of the subsequence
\[
	0 \to  \fM \cap \ker \exp_C  \to  \fM \to \cL \to 0
\]
is the short exact sequence
\[
	0 \to  \fM \cap \ker \exp_C  \to  \fM^{-} \to \C(K)_\tors \to 0.
\]

Comparing both, we find
\[
	\Fitt_{A[\Delta]} \frac{\cU}{\cL} = \Fitt_{A[\Delta]} \frac{\cU^+}{\cL^+} =
	\Fitt_{A[\Delta]} \frac{\U(\cO_K)^+}{\fM^+}.
\]
The ideal on the left equals $\Fitt_{A[\Delta]} \sqrt{\cL}/{\cL}$ by Proposition \ref{propdivclosure}, and the ideal on the right equals $\Fitt_{A[\Delta]} \H(\cO_K)^+$ by Corollary \ref{compareFitting}.
\end{proof}

In \cite{ANG&TAE} we have shown that $A_P \otimes_A \H(\cO_K)^+$ is not always trivial, unlike what one may expect by analogy with the Kummer-Vandiver conjecture. Combining this with Theorem \ref{thmdivclosure} we conclude

\begin{corollary}
There exist prime powers $q$ and monic irreducible $P \in \bF_q[T]$ so that $\sqrt{\cL}/{\cL}$ has nontrivial $P$-torsion.\qed
\end{corollary}

This settles Anderson's conjecture \cite[\S 4.12]{AND} in the negative. For example,
the prime 
\[
	P=T^9-T^6-T^4-T^3-T^2+1\,\, \in\,\, \bF_3[T]
\]
gives a counterexample \cite{ANG&TAE}.

\bigskip

\begin{void}
We now turn our attention to $P$-adic special values. We will combine Theorem \ref{thmdivclosure} with Anderson's $P$-adic construction of special points to prove Theorem \ref{thmeven}. Recall that $\fm$ denotes the maximal ideal of $\cO_{K,P}$.
Note that for every $N\geq 0$ the subgroup $\fm^N$ of $C(\cO_{K,P})$ is preserved by the Carlitz action. We denote the resulting $A$-module by $C(\fm^N)$.
\end{void}

\begin{proposition}\label{propPexp}
Assume $N\geq 2$. Then
\begin{enumerate}
\item  $\exp_\C X \in k[[X]]$ converges on $\fm^N$ and defines an isomorphism 
$\exp_{\C,P} \colon \fm^N \to \C(\fm^N)$ of topological $A[\Delta]$-modules;
\item the action of $A[\Delta]$ on $\C(\fm^N)$ extends uniquely to a continuous $A_P[\Delta]$-module structure on $\C(\fm^N)$;
\item the $A_P$-module $\C(\fm^N)$ is torsion-free.
\end{enumerate}
\end{proposition}

One can show that (2) also holds for $N=1$, but we will not need this.

\begin{proof} (2) and (3) follow immediately from (1). The proof of (1) is a matter of analyzing the $P$-adic valuation of the coefficients of $\exp_\C X$. Write $\exp_\C X$ as 
\[
	\exp_\C X = e_0 X + e_1 X^q + e_2 X^{q^2} + \cdots
\]
with $e_0=1$. Let $v$ be the valuation on $K_P$ normalized such that $v(P)=1$ (so that the valuation of a uniformizer of $\cO_{K,P}$ is $\frac{1}{q^d-1}$).
The functional equation $\exp_\C (TX) = T \exp_\C X + (\exp_\C X)^q$ implies that for all $n\geq 1$ we have
\[
	e_n   = \frac{1}{T^{q^n}-T} e_{n-1}^q.
\]
Note that 
\[
	v(T^{q^n}-T) = \begin{cases}
		1 & \text{if $d$ divides n,}\\
		0 & \text{otherwise,}
	\end{cases}
\]
from which we obtain by induction that
\[
	v(e_n) \geq -\frac{q^n}{q^d-1}
\]
for all $n$. This implies that $\exp_\C$ converges on $\fm^2$. Moreover, we see that for all nonzero $x\in \fm^2$ one has $v(x-\exp_\C x) > v(x)$, and hence that $\exp_\C$ induces the identity map on $\fm^i/\fm^{i+1}$ for all $i\geq 2$. Together with the functional equation for $\exp_\C$ this proves part (1) of the proposition.
\end{proof}

\begin{void}
Recall that $\cU$ denotes the image of $\U(\cO_K)$ in $\C(\cO_{K})$. The submodule
$\cU_2 := \cU \cap \C(\fm^2)$ is of finite index in $\cU$. We denote by $\widebar{\cU_2}$ its closure in $\C(\fm^2)$. This is an $A_P[\Delta]$-module which is torsion-free as $A_P$-module. The natural map
\[
	\alpha\colon A_P\otimes_A \cU_2 \to \widebar{\cU_2}
\]
is surjective. We will now show that $\alpha$ is an isomorphism, a statement analogous to \emph{Leopoldt's conjecture} for cyclotomic number fields (which is a theorem by Brumer \cite{BRU}). The main ingredient is a result on linear independence of $P$-adic Carlitz logarithms, which is shown by Vincent Bosser in an appendix to this paper. 
\end{void}

\begin{theorem}\label{thmLeopoldt} $\alpha$ is an isomorphism.
\end{theorem}

\begin{proof}
It suffices to show that $\alpha$ is injective. Note that $\cU^- \cap \C(\fm^2)=0$, so that the odd part of $\ker \alpha$ vanishes. So let $\chi\colon \Delta\to (A/PA)^\times$ be even. Since $e_\chi(A_P\otimes_A \cU_2)$ is free of rank one over $A_P$ it suffices to show that $e_\chi\widebar{\cU_2}$ is non-zero.

Let $x$ be a nonzero element of
\[
	e_\chi ( (A/PA) \otimes_{\bF_q} \cU_2).
\]
Let $u_1,\ldots,u_r$ be a basis of the $A$-module $\cU_2$. Then there are $\beta_i \in A/PA\otimes_{\bF_q} A$ so that
\[
	x = \beta_1 (1\otimes u_1) + \cdots + \beta_r (1\otimes u_r).
\]
By Proposition \ref{propPexp}, for each $i$ there is a unique $\eta_i \in \fm^2$ with $\exp_{\C,P} \eta_i = u_i$. Since the $\beta_i$ (or rather, their images in $K_P$) are algebraic over $k$, and since they are not all zero we have by
the Baker-Brumer theorem of Vincent Bosser (see the appendix) that
\[
	0 \neq \beta_1 \eta_1 + \cdots + \beta_r \eta_r \in \fm^2.
\]
Applying the $P$-adic Carlitz exponential we find a nonzero element $\exp_{\C,P} (\sum_i \beta_i \eta_i)$ in $e_\chi \widebar{\cU_2}$.
\end{proof}

\begin{corollary} $\widebar{\cU_2}^+$ is free of rank one over $A_P[\Delta]^+$. \qed
\end{corollary}

%
%
%

Let $\widebar{\cL_2} \subset \C(\cO_{K,P})$ be the topological closure of $\cL_2 := \cL \cap \fm^2$. Then $\widebar{\cL_2}$ is an $A_P[\Delta]$-submodule of $\widebar{\cU_2}$ with finite quotient. The following proposition is a crucial ingredient for the proof of Theorem \ref{thmeven}.

\begin{proposition}\label{propPAnd}
Let $\chi\colon \Delta \to A_P^\times$ be a homomorphism.  
Then 
\[
	e_\chi \widebar{\cL_2} = L_P(1,\chi) \cdot e_\chi \C(\fm^2)
\]
as $A_P$-submodules of $\C(\fm^2)$. In particular, $L_P(1,\chi)=0$ if and only if $\chi$ is odd.
\end{proposition}

\begin{proof}
Recall from section \ref{secnormalbasis} that $\lambda \in \cO_K$ denotes a fixed generator of $\C(K)[P]$. 
For $m\geq 1$ consider the series 
\[
	\fL_{m,P} := \sum_{\sigma \in \Delta} \sigma(\lambda)^m
	\left( \sum_{n \geq 0 } \sum_{a \in A_{+,n,\sigma}} \frac{1}{a}
	\right).
\]
in $\cO_{K,P}$.
Here $A_{+,n,\sigma}$ is the set of monic polynomials in $A$ of degree $n$ which reduce modulo $P$ to $\sigma\in (A/PA)^\times$. Anderson \cite[Proposition 12]{AND} has shown that these series converge $P$-adically to elements of $\fm^2$ satisfying the remarkable identities
\begin{equation}\label{eqAnderson}
	\exp_{\C,P} \fL_{m,P} = \exp_\C \fL_m \quad\text{ for all $m\geq 2$ }
\end{equation}
and
\begin{equation}\label{eqAnderson1}
	P \exp_{\C,P} \fL_{1,P} = P \exp_\C \fL_1.
\end{equation}
Note that these are identities in $\C(\cO_K)$, but that a priori their left-hand side is $P$-adic and lives in $\C(\fm^2)$ whereas their right-hand side is $\infty$-adic and lives in $\C(K_\infty)$.
By exactly the same reasoning as in Proposition \ref{propcircL} we have for all $m\geq 1$ and for all $\chi\colon \Delta\to A_P^\times$ that
\begin{equation}\label{eqLP1}
	e_\chi \fL_{m,P} = L_P(1,\chi) e_\chi \lambda^m.
\end{equation}

We will prove the claim of the proposition separately for odd and even $\chi$. 

If $\chi$ is odd then $e_\chi \bar{\cL}=0$. By (\ref{eqAnderson}) we have $\exp_{\C,P} \fL_{m,P} \in \bar{\cL}$ for all $m\geq2$. Since $\exp_{\C,P}$ defines an isomorphism of $A_P[\Delta]$-modules between $\fm^2$ and $\C(\fm^2)$ we find that $e_\chi \fL_{m,P}=0$ for all $m\geq 2$. On the other hand, $e_\chi \fm^2 \neq 0$ so there are $m\geq 2$ with $e_\chi \lambda^m \neq 0$. By (\ref{eqLP1}) we must have $L_P(1,\chi)=0$.

So now assume that $\chi$ is even. We have that $L_P(1,\chi) e_\chi \fm^2$ is generated as an $A_P$-module by
\[
	\left\{ L_P(1,\chi) e_\chi \lambda^m \colon m \geq 2 \right\}.
\]
By \cite[Proposition 9]{AND} the quotient of $\cL$ by the submodule generated by the $\exp_\C \fL_m$ with $m\geq 1$ is annihilated by $P-1$. Moreover, since $\chi$ is even we have $e_\chi \lambda =0$ and hence $e_\chi \fL_1=0$. We conclude that the $A_P$-module $e_\chi \bar{\cL}$ is  generated by
\[
	\left\{ e_\chi \fL_{m,P} \colon m \geq 2 \right\}.
\]
Comparing the two generating sets using (\ref{eqLP1}) and applying the isomorphism $\exp_{\C,P}$ 
between $\fm^2$ and $\C(\fm^2)$ we conclude $e_\chi \bar{\cL} =  L_P(1,\chi) \cdot e_\chi \C(\fm^2)$, as claimed. Since $e_\chi \bar{\cL} \neq 0$ we also find that $L_P(1,\chi)$ is non-zero.
\end{proof}

\begin{corollary}
For even $\chi$ we have
\[
	\length_{A_P} e_\chi(A_P\otimes_A \H(\cO_K)) = \length_{A_P} \frac{ e_\chi \widebar{\cU_2}}{ e_\chi\widebar{\cL_2}}.
\]
\end{corollary}

\begin{proof}
Consider the short exact sequence
\[
	0 \to \C(\fm)/\C(\fm^2) \to \C(\cO_{K,P}/\fm^2) \to \C(\cO_{K,P}/\fm) \to 0.
\]
On the one hand we have $\C(\cO_{K,P}/\fm) \cong \C(A/PA) \cong A/(P-1)A$, and on the other hand we have that $(\fm/\fm^2)^+=0$. From this we deduce that the finite $A$-module $\C(\cO_{K,P}/\fm^2)^+$ is $P$-torsion free.

Since the quotients $\cU^+/\cU_2^+$ and $\cL^+/\cL_2^+$ map injectively to $\C(\cO_{K,P}/\fm^2)^+$, we 
find an isomorphism of $A_P[\Delta]$-modules
\[
	 A_P\otimes_A\frac{\cU^+}{\cL^+} \isomto A_P\otimes_A\frac{\cU_2^+}{\cL_2^+}
\]
and together with Theorem \ref{thmLeopoldt} this yields an isomorphism of
$A_P[\Delta]$-modules.
\[
	A_P \otimes_A \frac{\cU^+}{\cL^+} \isomto \frac{\widebar{\cU_2}^+}{\widebar{\cL_2}^+}.
\]
The corollary now follows from Theorem \ref{thmdivclosure}.
\end{proof}

Together with Proposition \ref{propPAnd} this proves Theorem \ref{thmeven}:

\begin{theorem} Let $\chi\colon \Delta \to A_P^\times$ be even. Then $L_P(1,\chi)\neq 0$ and
\[
	\length_{A_P} e_\chi \left( A_P\otimes_A \H(\cO_K)\right)  +
	\length_{A_P} e_\chi \frac{\C(\fm^2)}{\bar{\cU}} = v_P( L_P(1,\chi) ).
\]\qed
\end{theorem}

\section*{Acknowledgements} 

The authors are grateful to the referees for various suggestions and corrections that improved the paper.

The second author would like to thank Ted Chinburg, David Goss, and Ambrus P\'al for numerous discussions related to this paper. He is supported by a VENI grant of the Netherlands Organisation for Scientific Research (NWO).

\end{document}


\maketitle

\section{Notation and statement of the Theorem}

We denote by $\NN=\{0,1,\ldots\}$ the set of nonnegative integers.
We write $A=\FF_q[T]$ and $k=\FF_q(T)$. Let $P\in A$ be an irreducible polynomial of degree $d\geq 1$, let
$k_P$ be the completion of $k$ at $P$, let $\CC_P$ be the completion of an algebraic closure of $k_P$, and let $\bar{k}\subset \CC_P$
be the algebraic closure of $k$ in $\CC_P$. We denote by $v=v_P$ the valuation on $\CC_P$ corresponding to $P$ normalized by $v_P(P)=1$, and
by $|\cdot|=|\cdot|_P$ the absolute value on $\CC_P$ defined by $|z|=q^{-v_P(z)}$. We denote by $\Phi : A\rightarrow k\{\tau\}$
the Carlitz module and by
\begin{equation}\label{Carlitz-series}
e(X)=\sum_{i\geq 0}\frac{X^{q^i}}{D_i}\in k[[X]]
\end{equation}
the Carlitz exponential series. Let $\rho := q^{-1/(q^d-1)}$ be the convergence radius of the series (\ref{Carlitz-series}),
and put
$$ D_{\rho}:=\{ z\in \CC_P \mid |z| < \rho\}.$$
We know that $e(z)$ is convergent in $\CC_P$ if and only if $z\in D_{\rho}$, and
that the series (\ref{Carlitz-series})
induces a bijection (the $P$-adic Carlitz exponential)
\begin{equation*}\label{e_P}
e : D_{\rho}\rightarrow D_{\rho}.
\end{equation*}
The inverse map will be denoted by $\Log$ ($P$-adic Carlitz logarithm).

The functions $e$ and $\Log$ satisfy the following properties:
\begin{equation*}\label{isometry}
\forall z\in D_{\rho},\ |e(z)|=|z|,
\end{equation*}
\begin{equation*}\label{exp}
\forall a\in A, \forall z\in D_{\rho},\ e(az)=\Phi_a(e(z)),
\end{equation*}
and
\begin{equation}\label{log}
\forall a\in A, \forall z\in D_{\rho},\ \Log(\Phi_a(z))=a\Log (z).
\end{equation}

The aim of this appendix is to give a proof of the following theorem:

\begin{theorem}\label{Baker}
Let $n\geq 1$, and let $\alpha_1,\ldots,\alpha_n$ be $n$ elements of $D_{\rho}\cap \bar{k}$. If the numbers
$\Log\alpha_1,\ldots,\Log\alpha_n$ are linearly independent over $k$, then they are linearly independent over $\bar{k}$.
\end{theorem}

This theorem is an analogue of a theorem of Brumer for $p$-adic logarithms in characteristic zero \cite{brumer}, which was
itself an analogue of a theorem of Baker \cite{bbaker} for usual complex logarithms.
Many generalizations and improvements of the results of Baker are known nowadays in characteristic zero. The interested reader might consult
\emph{e.g.} the first three chapters of the book \cite{w} as well as \cite{miw}, \cite{gaudron}, \cite{david-hirata} for an overview of known results.

In the framework of Drinfeld modules, it seems that only results for logarithms in $\CC_{\infty}$ (where $\infty$ is the zero of $1/T$) have been published yet.
The first analogue of Baker's theorem is due to Yu \cite{yu1} for an arbitrary Drinfeld module defined over $\bar{k}$, but under a separability
condition. This condition was removed for CM-Drinfeld modules in \cite{denis}, and
then in full generality independently in \cite{denis1} and \cite{yu2}. A (quantitative) generalization of these results is
established in \cite{bosser}. General results on \emph{algebraic} independence of logarithms have been obtained in \cite{papanikolas} and \cite{chang-papanikolas}.

Here we will consider for simplicity Carlitz-logarithms, and we will prove only a qualitative version of the so-called ``homogeneous case'' of the $P$-adic Baker's theorem. There is no doubt that one could obtain quantitative and nonhomogeneous\footnote{The
nonhomogeneous version of Baker's theorem would state that under the assumptions of Theorem~\ref{Baker}, the $n+1$ numbers
$1,\Log\alpha_1,\ldots,\Log\alpha_n$ are linearly independent over $\bar{k}$.}
results for an arbitrary Drinfeld module defined over $\bar{k}$, \emph{e.g.} using the methods of \cite{yu2} or \cite{bosser}.
However, the proof would be much more complicated.

We will give here a proof of Theorem~\ref{Baker} as self-contained as possible.
We will follow the exposition given in characteristic zero in \cite[Section~6.3]{miw}. This proof is very close
to the proof of \cite{denis}, but we use the method of ``interpolation determinant'' instead of using an auxiliary function constructed from
a ``Siegel's Lemma''. As in \cite{miw} or \cite{denis}, we will first show (in Section~\ref{weak-version}) that
it suffices to prove a weak form of Theorem~\ref{Baker}. Then, in Section~\ref{proof-baker}, we prove this weak version of Baker's theorem.

\section{A weak version of Baker's theorem}\label{weak-version}

In this section, we show that Theorem~\ref{Baker} follows from the following result:

\begin{theorem}\label{baker}
Let $n\geq 1$ be an integer,
let $\beta_1,\ldots,\beta_n$ be $n$ elements of $\bar{k}$ such that $1,\beta_1,\ldots,\beta_n$ are $k$-linearly independent, and such that
$|\beta_i|\leq 1$
for all $i$. Let further $\alpha_1,\ldots,\alpha_{n+1}$ be $n+1$ elements of $D_{\rho}\cap \bar{k}$.
Assume that the numbers $\Log\alpha_1,\ldots,\Log\alpha_{n+1}$ are linearly independent over $k$.
Then
$$\Log\alpha_{n+1}-(\beta_1\Log\alpha_1+\cdots+\beta_n\Log\alpha_n)\not= 0.$$
\end{theorem}

\medskip

\begin{proof}[Proof of Theorem~\ref{Baker}, assuming Theorem~\ref{baker}]
Suppose that Theorem~\ref{Baker} is false. Then there exist an integer $n\geq 1$ and elements $\alpha_1,\ldots,\alpha_n\in \bar{k}\cap D_{\rho}$
such that $\Log\alpha_1$, \ldots, $\Log\alpha_n$ are $k$-linearly independent but $\bar{k}$-linearly dependent.
Choose $n$ minimal satisfying this condition. We note that $n\geq 2$. Let $\beta_1,\ldots,\beta_n\in\bar{k}$, not all zero,
such that
\begin{equation}\label{linear1}
\sum_{i=1}^n\beta_i\Log\alpha_i=0.
\end{equation}
We claim that the minimality of $n$ implies that $\beta_1,\ldots,\beta_n$ are $k$-linearly independent. Indeed, suppose that they are not.
Then, by renumbering if necessary, we have a relation
\begin{equation}\label{linear2}
a_n\beta_n = \sum_{i=1}^{n-1}a_i\beta_i
\end{equation}
with $a_i\in A$ ($1\leq i\leq n$) and $a_n\not=0$. Then (\ref{linear1}) yields
\begin{equation}\label{linear3}
\sum_{i=1}^{n-1}(a_i\Log\alpha_n+a_n\Log\alpha_i)\beta_i=0.
\end{equation}
Now, by the functional equation (\ref{log}), we have
$$a_i\Log\alpha_n+a_n\Log\alpha_i = \Log(\Phi_{a_i}(\alpha_n)+\Phi_{a_n}(\alpha_i))=\Log\alpha'_i$$
with $\alpha'_i=\Phi_{a_i}(\alpha_n)+\Phi_{a_n}(\alpha_i) \in \bar{k}\cap D_{\rho}$.
Moreover, $\Log\alpha'_1,\ldots,\Log\alpha'_{n-1}$ are $k$-linearly independent. So these elements are
$\bar{k}$-linearly independent by minimality of $n$. The relation (\ref{linear3}) then implies
$\beta_1=\cdots=\beta_{n-1}=0$, hence $\beta_n=0$ by (\ref{linear2}),
which is a contradiction. Hence $\beta_1,\ldots,\beta_n$ are $k$-linearly independent as claimed.
Applying now Theorem~\ref{baker} to the numbers $\Log\alpha_1,\ldots,\Log\alpha_{n}$, we see that
(\ref{linear1}) cannot hold. This is a contradiction. 
\end{proof}

\section{Proof of Theorem~\ref{baker}}\label{proof-baker}

This section is devoted to the proof of Theorem~\ref{baker}. The argument follows the same lines
as a standard transcendence proof. We assume that the conclusion of Theorem~\ref{baker} is false and we try to derive a contradiction.
For this, we construct first (in Sections~\ref{ddef} and \ref{NV}) a certain non-zero algebraic number $\Delta$. This number is defined as a
minor of order $L$ of a certain matrix having $L$ rows. The fact that such a minor exists (\emph{i.e.} is not zero) rests on the use of
a zero estimate due to Denis.
In a second step (Section~\ref{UB}), we find an upper bound for $|\Delta|$
using analytic arguments (namely, a Schwarz Lemma).
In a third step (Section~\ref{LB}), we find a lower bound for $|\Delta|$, using the fact that $\Delta$ is algebraic and non zero:
we use here the so-called Liouville's inequality. The upper bound and the lower bound would be contradictory, and hence
we get the desired contradiction.

\bigskip
From now on, we suppose that we are given elements $\alpha_1,\ldots,\alpha_{n+1}$, $\beta_1,\ldots,\beta_n$ in $\bar{k}$ as in
Theorem~\ref{baker}. For $1\leq i\leq n+1$ we define $\lambda_i := \Log \alpha_i$
and we suppose that the following equality holds:
\begin{equation}\label{hypothesis}
\lambda_{n+1}=\beta_1\lambda_1+\cdots+\beta_n\lambda_n.
\end{equation}
We denote by $c_0,c_1,\ldots, c_4$ fixed real numbers depending only on $q, n$ and $\alpha_1,\ldots,\alpha_{n+1}$, $\beta_1,\ldots,\beta_n$
(such numbers will be called ``constants''). These numbers will appear during the proof and could be made explicit, but we did not carry out this
task.

\subsection{Construction of a matrix $\mathcal M$}\label{ddef}

We begin by defining positive integers $T_1$, $T_2$, $S$ as follows: $S$ is a constant chosen sufficiently large, and
\begin{equation}\label{parameters}
T_1=\lfloor q^{(1+1/n)S}/S^3\rfloor, \qquad T_2 = S^{2n} .
\end{equation}
Here, ``sufficiently large'' means that
all the inequalities that will occur in the proof below are satisfied.
In particular, we note that the choice of $S$ depends on the constants $c_0,\ldots,c_4$ defined before.
As for the choice (\ref{parameters}), it is imposed by the various constraints on $S$, $T_1$, $T_2$ that will appear during the proof (see Remark~\ref{final-remark}).

We introduce the following notation.
For any $\bm{s}=(s_1,\ldots,s_{n+1})\in A^{n+1}$, we define
$$\deg\bm{s}:= \max_{1\leq i \leq n+1}\{\deg s_i\}$$
and we define the two sets
\[
	 {\mathscr{T}}=\{(\tau_1,\ldots,\tau_n,t)\in\NN^n\times\NN \mid \tau_1+\cdots+\tau_n\leq T_1, 0\leq t\leq T_2 \},
\]
and
\[
	{\mathscr{S}}=\{\bm{s}\in A^{n+1} \mid \deg\bm{s}\leq S\}.
\]
For every $(\bm{\tau},t)\in {\cal T}$ with $\bm{\tau}=(\tau_1,\ldots,\tau_n)$, and for every $z_1,\ldots ,z_n \in \CC_P$ with $|z_i|\leq 1$ we define
$f_{(\bm{\tau},t)(z_1,\ldots,z_n)}$ by
\begin{equation}\label{functions}
f_{(\bm{\tau},t)}(z_1,\ldots,z_n)= z_1^{\tau_1}\ldots z_n^{\tau_n} \bigl(e(\sum_{i=1}^n\lambda_i z_i)\bigr)^t,
\end{equation}
and for any $\bm{s}\in {\cal S}$ we define the algebraic points 
\begin{equation}\label{points}
\zeta_{\bm{s}} = (s_1+s_{n+1}\beta_1, \ldots, s_n+s_{n+1}\beta_n) \in \bar{k}^n.
\end{equation}
Let $L$ be the cardinality of $\cal T$. We have
\begin{equation}\label{def-L}
L= \binom{T_1+n}{n}(T_2+1).
\end{equation}
We choose any ordering of the sets $\cal T$ and $\cal S$, and we consider the matrix
\begin{equation*}\label{matrix}
{\mathcal M} = \bigl(f_{(\bm{\tau},t)}(\zeta_{\bm{s}})\bigr)_{(\bm{\tau},t),\bm{s}}
\end{equation*}
where the rows are indexed by $(\bm{\tau},t)\in {\cal T}$ and the columns are indexed by $\bm{s}\in {\cal S}$.

We note that the entries of $\cal M$ are actually elements of $\bar{k}$. Indeed, writing $\zeta_{\bm{s}}=(\zeta_{\bm{s}}^{(1)},\ldots,\zeta_{\bm{s}}^{(n)})$
and using the hypothesis (\ref{hypothesis}), we have
$$ \sum_{i=1}^{n}\lambda_i \zeta_{\bm{s}}^{(i)} = \sum_{i=1}^{n} \lambda_i s_i  +  (\sum_{i=1}^{n} \lambda_i\beta_i) s_{n+1}
=\sum_{i=1}^{n+1} \lambda_is_i,$$
hence
\begin{equation*}\label{link}
e\bigl(\sum_{i=1}^{n}\lambda_i \zeta_{\bm{s}}^{(i)}\bigr) = \sum_{i=1}^{n+1}\Phi_{s_i}(\alpha_i) \in\bar{k}
\end{equation*}
and thus
\begin{equation}\label{matrixx}
{\mathcal M} = \Bigl((s_1+s_{n+1}\beta_1)^{\tau_1}\cdots (s_n+s_{n+1}\beta_n)^{\tau_n}\bigl(\sum_{i=1}^{n+1}\Phi_{s_i}(\alpha_i)\bigr)^t\Bigr)_{(\bm{\tau},t),\bm{s}}
\end{equation}
has algebraic entries.

\subsection{Construction of a minor $\Delta$}\label{NV}

Observe that by (\ref{parameters}) and (\ref{def-L}) we have $L\leq 2^{n+1}T_1^n T_2 < q^{(S+1)(n+1)}=\card{{\mathcal S}}$,
so the rank of the matrix $\mathcal M$ is at most $L$.
The aim of this section is to prove that this rank is exactly $L$.

\begin{proposition}\label{rank-L}
The rank of the matrix $\mathcal M$ is equal to $L$.
\end{proposition}

To prove this proposition we will use a zero estimate due to Denis \cite{denis}. To state his result, we need to introduce
further notation. If $N\geq 1$ is an integer, let us denote by $\End_{\FF_q-\mbox{\scriptsize lin}}(\GG_a^N)$ the $\FF_q$-algebra of
$\FF_q$-linear endomorphisms of $\GG_a^N$,
and by $F : (x_1,\ldots,x_N)\rightarrow (x_1^q,\ldots,x_N^q)$ the Frobenius map on $\GG_a^N$.
Recall that a $T$-module of dimension $N$ and rank $r\geq 0$ is a pair $G=(\GG_a^N,\Psi)$,
where $\Psi : A\rightarrow\End_{\FF_q-\mbox{\scriptsize lin}}(\GG_a^N)$ is an injective homomorphism of $\FF_q$-algebras such that
$\Psi(T)=a_0F^0+\cdots +a_rF^r$,
where $a_i\in M_N(\CC_P)$ ($0\leq i\leq r$), $a_r\not=0$, and the only eigenvalue of $a_0$ is $T$. When $N=1$ and $r=0$, we get
the ``trivial'' $T$-module, whose action on $\GG_a$ is the usual scalar action.

A morphism $\varphi : G_1\rightarrow G_2$ of $T$-modules is a morphism of algebraic groups that commutes with the actions of $A$.
It is called an isogeny if it is surjective with finite kernel.
Given a $T$-module $(\GG_a^N,\Psi)$, any connected algebraic subgroup $H$ of $\GG_a^N$ which is invariant under $\Psi_a$ for all $a\in A$ is called a
sub-$T$-module of $(\GG_a^N,\Psi)$. If $H$ is such a sub-$T$-module, we define $\deg H$ as the projective degree
of its Zariski closure $\overline{H}$ in $\PP^N$ via the embedding $\GG_a^N\hookrightarrow \PP^N$.
A $T$-module $(\GG_a^N,\Psi)$ is said to be simple if its only sub-$T$-modules are $\{0\}$ and $\GG_a^N$.

If $G_1=(\GG_a^{n_1},\Psi_1)$ and $G_2=(\GG_a^{n_2},\Psi_2)$ are two $T$-modules, we will denote by
$G_1\times G_2$ the $T$-module $(\GG_a^{n_1+n_2},\Psi_1\times\Psi_2)$, where $\Psi_1\times\Psi_2$ is the diagonal action
on $\GG_a^{n_1}\times\GG_a^{n_2}$. We define similarly a product of finitely many $T$-modules.
In the zero estimate below, we will consider a product of $T$-modules $G=G_1^{n_1}\times\cdots\times G_m^{n_m}$, where
$G_i=(\GG_a^{d_i},\Psi_i)$ is a $T$-module of dimension $d_i$. We will denote by $\CC_P[X_{i,1},\ldots,X_{i,n_i d_i}]$
the coordinate ring of $G_i^{n_i}$, and by ${\bf X}_i=(X_{i,1},\ldots,X_{i,n_i d_i})$ its set of variables.
If $Q\in\CC_P[{\bf X}_1,\ldots,{\bf X}_m]$ is any polynomial, we will denote by $\deg_{{\bf X}_i}Q$ its partial degree with respect to ${\bf X}_i$.
Finally, if $(\GG_a^N,\Psi)$ is a $T$-module and $\Gamma=\{\gamma_1,\ldots,\gamma_g\}$ is a finite set of points of $\CC_P^N$,
we denote, for any integer $S\geq 0$,
$$ \Gamma(S) := \left\{\sum_{i=1}^g\Psi_{s_i}(\gamma_i) \mid s_1,\ldots,s_g\in A, \deg s_i\leq S\right\}.$$

We can now state the result of Denis.

\begin{theorem}[Zero estimate]\label{zero-estimate}
Let $G_i=(\GG_a^{d_i},\Psi_i)$ ($1\leq i\leq m$) be $m$ $T$-modules of dimension $d_i$ and rank $r_i\geq 0$.
Suppose that $G_i$ is simple for all $i$ and that $G_1,\ldots,G_m$ are pairwise non-isogeneous.
Let $n_1,\ldots,n_m$ be positive integers, and put $G=G_1^{n_1}\times\cdots\times G_m^{n_m}$, $N=d_1n_1+\cdots +d_mn_m$.
Let further $S\geq 0$ be an integer and $\Gamma$ a finite set of points in $\CC_P^N$. Suppose that there exists a non-zero
polynomial $Q\in\CC_P[{\bf X}_1,\ldots,{\bf X}_m]$ which vanishes on $\Gamma(S)$, and such that $\deg_{X_i} Q\leq L_i$ ($1\leq i\leq m$), where $L_i\in\ZZ_{>0}$.
Then there exist sub-$T$-modules $H_i$ of $G_i^{n_i}$ such that the product $H=H_1\times\cdots\times H_m$
is distinct from $\GG_a^N$ and such that
\begin{multline*}\label{zero}
\card\bigl((\Gamma(S-N+1)+H)/H\bigr)\,\cdotp\, \frac{(\dim H)!}{(\dim H_1)!\ldots (\dim H_m)!}\,\cdotp\, \prod_{i=1}^{m}\deg H_i\\
\leq \frac{N!}{(n_1d_1)!\ldots (n_md_m)!}\,\cdotp\,\prod_{i=1}^{m}\deg G_i^{n_i}\, . \,\prod_{i=1}^m(\kappa_iL_i)^{d_in_i-\dim H_i}, 
\end{multline*}
where $\kappa_i=q^{r_i(N-1)}$.
\end{theorem}

\begin{proof} See \cite[Theorem~2]{denis}.
\end{proof}

\begin{proof}[Proof of Proposition~\ref{rank-L}]
Suppose that $\rank({\mathcal M})<L$. Then there exists a non-trivial linear combination of the rows of $\mathcal M$ which vanishes,
that is, by (\ref{matrixx}),
$$ \sum_{(\bm{\tau},t)}\lambda_{(\bm{\tau},t)}(s_1+s_{n+1}\beta_1)^{\tau_1}\ldots (s_n+s_{n+1}\beta_n)^{\tau_n}
\bigl(\sum_{i=1}^{n+1}\Phi_{s_i}(\alpha_i)\bigr)^t=0$$
for all $\bm{s}\in {\cal S}$. 
Define the following points of $\overline{k}^{n+1}$ :
$$\gamma_1=(1,0,\ldots,0,\alpha_1),\ldots, \gamma_n=(0,\ldots,0,1,\alpha_n), \gamma_{n+1}=(\beta_1,\ldots,\beta_n,\alpha_{n+1}),$$
and consider the $T$-module $G=G_1^n\times G_2$, where $G_1$ is the trivial $T$-module of dimension $1$ and where $G_2=(\GG_a,\Phi)$
is the Carlitz module. We note that $G_1$ and $G_2$ are simple and non-isogeneous. Moreover, if $\Psi$ denotes the action of $A$ on $G$,
we have, for all $\bm{s}\in {\cal S}$,
$$ \Psi_{s_1}(\gamma_1)+\cdots + \Psi_{s_{n+1}}(\gamma_{n+1})=(s_1+s_{n+1}\beta_1,\ldots, s_n+s_{n+1}\beta_n,
\sum_{i=1}^{n+1}\Phi_{s_i}(\alpha_i)).$$
Thus, we see that the non zero polynomial
$$Q({\bf X}_1,\ldots,{\bf X}_n,Y)= \sum_{(\bm{\tau},t)}\lambda_{(\bm{\tau},t)}{\bf X}_1^{\tau_1}\ldots {\bf X}_n^{\tau_n}Y^t$$
vanishes on $\Gamma(S)$. Moreover, $\deg_{X}\leq T_1$ and $\deg_YQ\leq T_2$.
Applying Theorem~\ref{zero-estimate}, we find that there is a sub-$\CC_P$-vector space $H_1\subset \CC_P^n$ and a connected
algebraic subgroup $H_2\subset\GG_a$ such that $H=H_1\times H_2\not=\GG_a^{n+1}$, and
\begin{equation}\label{bbb}
\card\bigl((\Gamma(S-n)+H)/H\bigr) \leq c_0 T_1^{n-\dim H_1}T_2^{1-\dim H_2},
\end{equation}
where $c_0$ is a constant depending on $n$ and $q$.
Suppose first that $H_2=\{0\}$. Then clearly, by considering the last coordinate only,
\begin{equation}\label{ccc}
\card\bigl((\Gamma(S-n)+H)/H\bigr)\geq \card\{\sum_{i=1}^{n+1}\Phi_{s_i}(\alpha_i)\mid \deg s_i\leq S-n\}.
\end{equation}
Since $\sum_{i=1}^{n+1}\Phi_{s_i}(\alpha_i)=e(\sum_{i=1}^{n+1}s_i\lambda_i)$ and $e:D_{\rho}\rightarrow D_{\rho}$ is injective,
the $k$-linear independence of $\lambda_1,\ldots,\lambda_{n+1}$ implies that the points
$\sum_{i=1}^{n+1}\Phi_{s_i}(\alpha_i)$, $\deg s_i\leq S-n$, are all distincts. 
Hence (\ref{bbb}) and (\ref{ccc}) yield
$$ q^{(S-n+1)(n+1)}\leq c_0T_1^n T_2, $$
which is a contradiction by the choice of the parameters (\ref{parameters}) (recall that $S$ is chosen sufficiently large).

Suppose now that $H_2=\GG_a$. Since $H=H_1\times H_2\not=\GG_a^{n+1}$, we have $\dim H_1\leq n-1$.
Let us define
\begin{eqnarray*}
\Delta[S-n]&:=&\{(s_1+s_{n+1}\beta_1,\ldots, s_n+s_{n+1}\beta_n) \mid \deg s_i\leq S-n\} \\
&=&\{\sum_{i=1}^{n+1}s_i\xi_i \mid \deg s_i\leq S-n\},
\end{eqnarray*}
where $(\xi_1,\ldots,\xi_n)$ is the canonical basis of $\CC_P^n$ and $\xi_{n+1}=(\beta_1,\ldots,\beta_n)$.
We have
$$ \card\bigl((\Gamma(S-n)+H)/H\bigr)= \card\bigl((\Delta[S-n]+H_1)/H_1\bigr).$$
Put $r=\dim(\CC_P^n/H_1)$ and $\rho=\dim_k(<\xi_1,\ldots,\xi_{n+1}>_k+H_1/H_1)$, where $<\xi_1,\ldots,\xi_{n+1}>_k$ denotes
the sub-$k$-vector space of $\CC_P^n$ spanned by $\xi_1,\ldots,\xi_{n+1}$. We claim that $\rho\geq r+1$.
Indeed, there are already at least $r$ elements among $\xi_1,\ldots,\xi_n$ that are $k$-linearly independent modulo $H_1$, say
$\xi_{i_1},\ldots,\xi_{i_r}$. Hence $\rho\geq r$. Suppose that we have $\rho=r$. 
Then $\dim_k(<\xi_1,\ldots,\xi_{n}>_k+H_1/H_1)=r$ and 
the vector space $\CC_P^n/H_1$ is defined over $k$, hence $H_1$ is also defined over $k$. Now,
$\xi_{i_1},\ldots,\xi_{i_r},\xi_{n+1}$ are $k$-linearly dependent modulo $H_1$, so we can write
$\xi_{n+1}=\sum_{\ell=1}^r\mu_{i_{\ell}}\xi_{i_{\ell}}+h$, where $\mu_{i_{\ell}}\in k$ and $h\in H_1$. Thus
we have $(\beta_1-\mu_1,\ldots,\beta_n-\mu_n)\in H_1$, where we have put $\mu_i=0$ if $i\not=\{i_1,\ldots,i_r\}$.
But then the point $(\beta_1-\mu_1,\ldots,\beta_n-\mu_n)$ is contained in a hyperplane of $\CC_P^n$ defined over $k$,
which contradicts the $k$-linear independence of $1,\beta_1,\ldots,\beta_n$. Hence
$\rho\geq r+1$, as claimed.
It readily follows from this that
$$\card\bigl((\Gamma(S-n)+H)/H\bigr)=\card\bigl((\Delta[S-n]+H_1)/H_1\bigr)\geq q^{(S-n+1)(r+1)},$$
hence (\ref{bbb}) yields
$$ q^{S-n+1} \leq c_0 (T_1/q^{S-n+1})^r. $$
Since $T_1\geq q^{S-n+1}$, we get $q^{S-n+1} \leq c_0 (T_1/q^{S-n+1})^n$, which again contradicts the choice of the parameters (\ref{parameters}).
\end{proof}

It follows from Proposition~\ref{rank-L} that there is a $L\times L$ minor of $\mathcal M$ which is not zero.
We choose $L$ column indices $\bm{s}_1,\ldots,\bm{s}_L$ of $\cal S$ such that the corresponding minor does not vanish.
For ease of notation, in the sequel we rename the points $\zeta_{\bm{s}_1},\ldots,\zeta_{\bm{s}_L}$ as $\zeta_1,\ldots,\zeta_L$,
and similarly, we rename the functions $\{f_{(\bm{\tau},t)}\mid (\bm{\tau},t)\in {\cal T}\}$ as $f_{\lambda}$, $1\leq \lambda\leq L$.
We set
$$ \Delta := \det(f_{\lambda}(\zeta_{\mu}))_{1\leq \lambda,\mu\leq L}.$$

\subsection{Upper bound for $|\Delta|$}\label{UB}

Recall that $\rho=q^{-1/(q^d-1)}$ is the radius of convergens of the $P$-adic Carlitz exponential. In this section we will prove:

\begin{proposition}\label{upper-bound}
Let $E_0 = \rho/\max_{1\leq i\leq n}\{|\lambda_i|\}$. Then we have
$$ \log_q |\Delta| \leq -\frac{n}{2e}L^{1+1/n}\log_q E_0.$$
\end{proposition}
(Here $e$ denotes Euler's number $e \approx 2.7$).

We will need the ultrametric version of the Schwarz lemma.
If $R>0$ is a real number, we define $D(0,R)=\{z\in\CC_P\mid |z|<R\}$. We will say that a function
$\psi:D(0,R)\rightarrow \CC_P$ is analytic in $D(0,R)$ if we can write $\psi(z)=\sum_{n\geq 0}a_nz^n$
for all $z\in D(0,R)$. In that case, if $r$ is any real such that $0<r<R$, we define
$$|\psi|_r := \sup\{|\psi(z)|\mid |z|\leq r\}. $$

\begin{lemma}[Schwarz Lemma] \label{schwarz}
Let $0<r\leq R$ be two positive real numbers in the group of values $|\CC_P^{\times}|=q^{\QQ}$, and let $\psi$
be a non zero analytic function in a disk containing strictly $D(0,R)$. If $M=\ord_{z=0}\psi(z)$, then
$$ |\psi|_r\leq \left(\frac{r}{R}\right)^M |\psi|_R.$$
\end{lemma}

\begin{proof} Define $\varphi(z)=z^{-M}\psi(z)$. By the maximum modulus principle,
we have $|\varphi|_r=r^{-M}|\psi|_r$ and $|\varphi|_R=R^{-M}|\psi|_R$. The lemma now follows from the obvious inequality
$|\varphi|_r\leq |\varphi|_R$. 
\end{proof}

\begin{proof}[Proof of Proposition~\ref{upper-bound}]
We introduce the following analytic function in one variable, for $|z|$ small :
\begin{equation*}\label{function-D}
D(z) = \det(f_{\lambda}(z\zeta_{\mu}))_{1\leq \lambda,\mu\leq L}.
\end{equation*}
We claim that this function is actually analytic in the disk $D(0,E_0)$. Indeed, if $\zeta_{\mu}=(\zeta_{\mu,1},\ldots,\zeta_{\mu,n})$,
one readily checks from the definition of the points $\zeta_{\mu}$ (see (\ref{points})) and from the fact that $|\beta_i|\leq 1$ for all $i$,
that $|\zeta_{\mu,i}|\leq 1$ for all $i$, hence
$|\sum_{i=1}^n\lambda_i\zeta_{\mu,i}|\leq \max_{i}|\lambda_i| $ and thus $z(\sum_{i=1}^n\lambda_i\zeta_{\mu,i}) \in D_{\rho}$
if $|z|<E_0$. Coming back to the definition of the functions $f_{\lambda}$ (see (\ref{functions})), this proves the claim. We have moreover,
if $\lambda$ corresponds to the $n+1$-tuple $(\tau_1,\ldots,\tau_n,t)\in {\cal T}$ :
\begin{equation}\label{maj}
|f_{\lambda}(z\zeta_{\mu})|\leq |z|^{\tau_1+\cdots+\tau_n}\rho^t \quad\mbox{for all}\ z\in D(0,E_0).
\end{equation}
We apply now the Schwarz Lemma to the function $D$ with $r=1$ and $R\in q^{\QQ}$ such that $r\leq R<E_0$. We obtain
$$ |\Delta|=|D(1)|\leq R^{-M} |D|_R, $$
where $M=\ord_{z=0}D(z)$. We deduce from (\ref{maj}) the estimate $|D|_R\leq R^{LT_1}$, hence 
\begin{equation}\label{deltabound}
|\Delta|\leq R^{-M+LT_1}.
\end{equation}

Let us estimate from below the multiplicity $M$ at $0$ of the function $D(z)$. We follow here almost verbatim \cite{miw}, Lemmas~6.4 and 6.5.
By multilinearity
of the determinant and by expanding each function $f_{\lambda}$ at $(0,\ldots,0)$ as
$f_{\lambda}(\bm{z})=\sum_{\bm{i}}f_{\lambda,\bm{i}}\bm{z}^{\bm{i}}$ (where $\bm{z}^{\bm{i}}$ means as usual $z_1^{i_1}\ldots z_n^{i_n}$ when
$\bm{z}=(z_1,\ldots,z_n)$ and $\bm{i}=(i_1,\ldots,i_n)$),
we see that we may assume that each entry of $D(z)$ is a monomial of the form $z^{\|\bm{i}\|}\zeta_{\bm{\mu}}^{\bm{i}}$,
where $\|\bm{i}\|:=i_1+\cdots +i_n$. In that case, we have the common factor $z^{\|\bm{i}\|}$ in each row indexed by
$\bm{i}$. Moreover, we may assume that two different rows correspond to two different indices $\bm{i}$, because otherwise the two rows
would be identical and the corresponding determinant would be zero. We deduce from this that the vanishing order $M$ of $D(z)$ at $0$ 
satisfies
$$ M \geq \Theta_n(L):=\min\{\|\bm{i_1}\|+\cdots +\|\bm{i_L}\|\},$$
where the minimum runs over all the $L$-tuples $(\bm{i_1},\ldots,\bm{i_L})\in \NN^n\times\cdots\times\NN^n$ which are pairwise distinct.
Lemma~6.5 of \cite{miw} yields the estimate $\Theta_n(L)>(n/e)L^{1+1/n}$ as soon as $L\geq (4n)^{2n}$. By the choice (\ref{parameters}), this
latter condition is satisfied. Combining (\ref{deltabound}) with this estimate for $M$ we conclude 
$$ \log_q |\Delta| \leq (LT_1-\frac{n}{e}L^{1+1/n})\log_q R\leq -\frac{n}{2e}L^{1+1/n}\log_q R$$
(since $T_1\leq (n/2e)L^{1/n}$ by (\ref{parameters})).
Letting now $R$ tend to $E_0$, we obtain the proposition. 
\end{proof}

\subsection{Lower bound for $|\Delta|$ and conclusion.}\label{LB}

In this section we prove the following lower bound for $|\Delta|$, from which we derive the desired contradiction.

\begin{proposition}\label{lower-bound}
The following inequality holds:
$$ \log_q|\Delta| \geq -c_2 LS(T_1+T_2q^S).$$
\end{proposition}

To prove this proposition we will use the "Liouville's inequality".
We will need the notion of height of an algebraic point of $\PP^N(\bar{k})$. We recall for convenience the definition and
the basic properties we will use.

If $(\xi_0:\xi_1:\cdots :\xi_N)$ is a point of $\PP^N(\bar{k})$, we define its height by the formula
\begin{equation}\label{def-height}
h(\xi_0:\xi_1:\cdots :\xi_N) = \frac{1}{[K:k]}\sum_{w\in M_K}d_w \max\{-w(\xi_0),\ldots,-w(\xi_N)\},
\end{equation}
where $K/k$ is any finite extension such that $\xi_0,\ldots,\xi_N\in K$, $M_K$ is the set of non trivial places of $K$,
$d_w$ is the degree over $\FF_q$ of the residue class field at $w$, and the valuation $w$ is normalized so that
$w(K^{\times})=\ZZ$.
The properties of the valuations show that this definition does not depend on the choice of $K$ containing $\xi_0,\ldots,\xi_N$,
and the product formula shows that it is independent of the chosen projective coordinates $(\xi_0,\ldots,\xi_N)$ for the point.
If $\xi$ is an element of $\bar{k}$, we define $h(\xi)$ by $h(\xi):=h(1:\xi)$.
 
\begin{lemma}\label{hheight}
Let $f\in A[X_1,\ldots,X_{n+1},Y_1,\ldots,Y_n]$ be any non zero polynomial.
We have
$$ h(f(\alpha_1,\ldots,\alpha_{n+1},\beta_1,\ldots,\beta_n))\leq c_1 \deg f + \delta(f), $$
where $c_1=h(1:\alpha_1:\cdots:\alpha_{n+1}:\beta_1:\cdots:\beta_n)$ and where $\delta(f)$ denotes
the maximum of the degrees (in $T$) of the coefficients of $f$.
\end{lemma}

\begin{proof}
 Let $K=k(\alpha_1,\ldots,\alpha_{n+1},\beta_1,\ldots,\beta_n)$. One easily checks that for any place $w$ of $K$,
one has
\begin{multline*}
\max\{0,-w(f(\alpha_1,\ldots,\alpha_{n+1},\beta_1,\ldots,\beta_n))\}\\
\leq \deg f\, \max\{0,-w(\alpha_1),\ldots,-w(\alpha_{n+1}),-w(\beta_1),\ldots,-w(\beta_n)\} + c_w\\
\end{multline*}
with $c_w=0$ if $w\nmid\infty$ and $c_w=e_w\delta(f)$ if $w\mid\infty$ (here $e_w$ is the ramification index at $w$).
The lemma follows from this and the definition of the height (\ref{def-height}). 
\end{proof}

\begin{corollary}[Liouville's inequality]\label{liouville}
With the notations of Lemma~\ref{hheight}, we have, if $f(\alpha_1,\ldots,\alpha_{n+1},\beta_1,\ldots,\beta_n)\not=0$,
$$\log_q |f(\alpha_1,\ldots,\alpha_{n+1},\beta_1,\ldots,\beta_n)|\geq - [K:k] (c_1 \deg f + \delta(f)),$$
where $K=k(\alpha_1,\ldots,\alpha_{n+1},\beta_1,\ldots,\beta_n)$.
\end{corollary}

\begin{proof}
 Put $\xi:=f(\alpha_1,\ldots,\alpha_{n+1},\beta_1,\ldots,\beta_n)$.
Since $(1:\xi)=(\xi^{-1}:1)$, we have $h(\xi)=h(\xi^{-1})$.
Hence 
$$ \log_q|\xi^{-1}|=-v(\xi^{-1}) \leq \max\{0,-v(\xi^{-1})\}\leq [K:k] h(\xi^{-1})=[K:k] h(\xi).$$
Now, Lemma~\ref{hheight} yields the result.
\end{proof}

\begin{proof}[Proof of Proposition~\ref{lower-bound}.]
From the definition of $\Delta$ and the expression (\ref{matrixx}), we see that we can write
$\Delta=f(\alpha_1,\ldots,\alpha_{n+1},\beta_1,\ldots,\beta_n)$,
where $f$ is a polynomial of $A[X_1,\ldots,X_{n+1},Y_1,\ldots,Y_n]$ of the form
$$ f(X_1,\ldots,X_{n+1},Y_1,\ldots,Y_n)=
\det\Bigl((s_1+s_{n+1}Y_1)^{\tau_1}\ldots (s_n+s_{n+1}Y_n)^{\tau_n}\bigl(\sum_{i=1}^{n+1}\Phi_{s_i}(X_i)\bigr)^t\Bigr)_{(\bm{\tau},t),\bm{s}}$$
($(\bm{\tau},t)$ runs over all the elements of $\cal T$ and $\bm{s}$ runs over a subset of ${\cal S}$ of cardinality $L$).
We have, since $\deg(\Phi_{s_i}(X_i))=q^{\deg s_i}\leq q^S$ :
\begin{equation*}\label{degree-estimate}
\deg_Xf\leq LT_2 q^S \quad \mbox{and}\quad \deg_Yf\leq LT_1.
\end{equation*}
Moreover, the coefficients of each polynomial $\Phi_{s_i}(X_i)$ are elements of $A$ of degree in $T$ at most $q^{\deg s_i}\deg s_i\leq Sq^S$,
hence the coefficients of $f$ have a degree in $T$ at most $L(T_1S+T_2Sq^S)$. It follows from these estimates and from Liouville's inequality
that
$$ \log_q|\Delta| \geq -c_2 LS(T_1+T_2q^S).$$
\end{proof}

\begin{proof}[End of the proof of Theorem~\ref{baker}.] By Propositions \ref{upper-bound} and \ref{lower-bound}, we have:
$$ \frac{n}{2e}L^{1+1/n}\log_q E_0 \leq c_2 LS(T_1+T_2q^S)$$
or
$$ L^{1/n}\leq c_3 S(T_1+T_2q^S). $$
Since by (\ref{def-L}) we have $L\geq T_1^nT_2/n!$, we obtain
$$ T_1 T_2^{1/n}\leq c_4 (T_1S+ T_2Sq^S). $$
But this inequality contradicts the choice of the parameters (\ref{parameters}).
Thus the assumption (\ref{hypothesis}) was false, which completes the proof of Theorem~\ref{baker}.
\end{proof}

\begin{remark}\label{final-remark}
If we keep track of all the inequalities that the parameters $S, T_1$ and $T_2$ should satisfy
in order that the above proof works, we see that these parameters have to be sufficiently large and should satisfy the following three conditions:
(i) $T_1\geq q^{S-n+1}$, (ii) $c_0T_1^nT_2 < q^{(S-n+1)(n+1)}$ and (iii) $c_4 (T_1S+ T_2Sq^S) < T_1 T_2^{1/n}$. The definition (\ref{parameters})
has been chosen to fulfill these conditions.
\end{remark}